\documentclass[reqno,11pt]{amsart}

\usepackage{amsthm,amsfonts, amssymb, amscd}

\usepackage{tikz} \usetikzlibrary{matrix,arrows}
\usetikzlibrary{positioning}

\usepackage{hyperref}
\hypersetup{%
   linktoc=page
}

\let\oldtocsection=\tocsection \let\oldtocsubsection=\tocsubsection
\renewcommand{\tocsection}[2]{\hspace{0em}\oldtocsection{#1}{#2}}
\renewcommand{\tocsubsection}[2]{\hspace{2em}\oldtocsubsection{#1}{#2}}

\theoremstyle{plain}
\newtheorem{theorem}{Theorem}[section] 
\newtheorem{lemma}[theorem]{Lemma}
\newtheorem{corollary}[theorem]{Corollary} 
\newtheorem{proposition}[theorem]{Proposition}

\theoremstyle{definition}

\newtheorem{definition}[theorem]{Definition}

\theoremstyle{remark} 
\newtheorem{remark}[theorem]{Remark}
\newtheorem{example}[theorem]{Example}

\numberwithin{equation}{section}

\newcommand{\A}{A}
\newcommand{\rs}{S}
 
\newcommand{\R}{\mathbb{ R}} 
\newcommand{\Z}{\mathbb{ Z}} 
\newcommand{\PP}{\mathbb{ P}} 
\newcommand{\Q}{\mathbb{ Q}} 

\newcommand{\Char}{\operatorname{Char}}

 \newcommand{\ft}{{\mathfrak t}}

\newcommand{\ga}{\alpha} \newcommand{\gb}{\beta}
\newcommand{\gd}{\delta} \newcommand{\gD}{\Delta}
 \newcommand{\gre}{\epsilon}
\renewcommand{\gg}{\gamma} 
 \newcommand{\gl}{\lambda}

\newcommand{\gP}{\Phi}  \newcommand{\gs}{\sigma}
 
 \newcommand{\gz}{\zeta}

 \newcommand{\ch}{\mathcal{H}}

\newcommand{\ct}{\mathcal{T}}

 \newcommand{\vp}{\mathbf{p}}
\newcommand{\vq}{\mathbf{q}}
\newcommand{\vr}{\mathbf{r}} \newcommand{\vs}{\mathbf{s}}
\newcommand{\vt}{\mathbf{t}}

\newcommand{\IP}[2]{\langle#1 , #2\rangle} 

\renewcommand{\bar}[1]{\overline{#1}}

\newcommand{\wh}{\widehat}

\newcommand{\zind}[1]{{{#1}^{\Z}}}

\newcommand{\isnd}[1]{{#1}^{\ddag}}

\newcommand{\conea}{\operatorname{Cone}_{\A}}
\newcommand{\conez}{\operatorname{Cone}_{\Z}}

\newcommand{\coner}{\operatorname{Cone}_{\R}}

\newcommand{\pam}{I(x^{-1})}
\newcommand{\pcuk}{\Phi_{\operatorname{cur}}^{\operatorname{KL}}}
\newcommand{\ptak}{\Phi_{\tan}^{\operatorname{KL}}}

\newcommand{\cur}{_{\operatorname{cur}}}
 \newcommand{\oth}{\operatorname{oth}}

\newcommand{\gew}{_{\geq w}}
\newcommand{\zgew}{_{z \geq w}}

\newcommand{\upm}{U_P^-(x)} 
\newcommand{\uup}{U_P^-(x)\cap U}
\newcommand{\uub}{U^-(x)\cap U}

\newcommand{\pupm}{x\gP_P^-} 
\newcommand{\puup}{x\gP_P^-\cap \gP^+}
\newcommand{\puub}{x\gP^-\cap \gP^+}

\newcommand{\st}{\subseteq}

\newcommand{\xw}{X^w}

\newcommand{\kxw}{Y_x^w}

\newcommand{\vrr}{V_{\R}}
\newcommand{\vqq}{V_{\Q}}

\newcommand{\xgw}{_{x \geq w}}
\newcommand{\zgw}{_{z \geq w}}

\newcommand{\ua}{^{\A}}
\newcommand{\uq}{^{\Q}}
\newcommand{\uz}{^{\Z}}
\newcommand{\ur}{^{\R}}
\newcommand{\uia}{^{\scriptscriptstyle\uparrow \A}}
\newcommand{\uiz}{^{\scriptscriptstyle\uparrow \Z}}
\newcommand{\uiq}{^{\scriptscriptstyle\uparrow \Q}}
\newcommand{\uis}{^{\ddag}}

\newcommand{\pcu}{\Phi_{\operatorname{cur}}}
\newcommand{\pta}{\Phi_{\tan}}

\newcounter{myenumi}
\renewcommand{\themyenumi}{$(\arabic{myenumi})$}
\newenvironment{myenumerate}{%
\setlength{\parindent}{10pt}
\setcounter{myenumi}{0}
\renewcommand{\item}{
\par
\refstepcounter{myenumi}
\makebox[2.5em][c]{\themyenumi}
}
}{
\par
\noindent
\ignorespacesafterend
}

\begin{document}
\parskip=4pt \baselineskip=14pt

\title[Tangent space and $T$-invariant curves]{Tangent spaces and
  $T$-invariant curves\\ of Schubert varieties}


\author{William Graham} \address{ Department of Mathematics,
  University of Georgia, Boyd Graduate Studies Research Center,
  Athens, GA 30602 } \email{wag@uga.edu}

\author{Victor Kreiman} \address{ Department of Mathematics and
  Physics, University of Wisconsin - Parkside, Kenosha, WI 53140 }
\email{kreiman@uwp.edu}


\date{\today}
\begin{abstract}
  The set of $T$-invariant curves in a Schubert variety through a
  $T$-fixed point is relatively easy to characterize in terms of its
  weights, but the tangent space is more difficult. We prove that the
  weights of the tangent space are contained in the rational cone
  generated by the weights of the $T$-invariant
  curves.  In simply laced types, this remains true if ``rational" is replaced
  by ``integral".  We also obtain conditions under which every weight
  of the tangent space is the weight of a $T$-invariant curve, as well as
  a smoothness criterion.  The results rely on equivariant $K$-theory,
  as well as the study of different notions of decomposability of roots.
\end{abstract}

\maketitle

\tableofcontents

\section{Introduction}\label{s.intro}

Let $X$ be a Schubert variety in the generalized flag variety $G/P$,
and let $x$ be a $T$-fixed point of $X$. In this paper we study two
spaces: the tangent space to $X$ at $x$, and the span of tangent lines
to $T$-invariant curves through $x$. These spaces are characterized by
their weights, which we denote by $\pta$ and $\pcu$ respectively.  

In \cite{LaSe:84}, Lakshmibai and Seshadri obtained a formula for
$\pta$ in type $A$ and gave criteria based on this formula for $X$ to
be smooth at $x$.  Lakshmibai then gave detailed formulas for $\pta$
in all classical types \cite{Lak:95}, \cite{Lak:00}, and
\cite{Lak2:00} (see also \cite[Chapter 5]{BiLa:00}).  In
\cite{Car:94}, Carrell and Peterson gave a formula for $\pcu$ and
discovered a test for determining whether $X$ is rationally smooth at
$x$.

The Carrell-Peterson formula for $\pcu$ has certain advantages: it
holds in all types, is type-independent and relatively simple, and has
clear connections to combinatorics.  The purpose of this paper is to
study the relationship between $\pta$ and $\pcu$, with an eye toward
replacing $\pta$ by the computationally simpler $\pcu$ in certain
applications. It is known that $\pcu\st \pta$, with equality in type
$A$.  Our main result is the following:

\begin{theorem}\label{t.main}
  $\pta\st \conea \pcu$.
\end{theorem}

In this theorem and throughout this section $A=\Q$ (except when
referring to root systems of ``type $A$''); in simply laced types all
results hold for $A=\Z$ as well.  By $\conea \pcu$, we mean the set of
all nonnegative $\A$-linear combinations of elements of $\pcu$.

This result is new even in classical types (besides type $A$).  Indeed,
Lakshmibai's formulas for the tangent spaces are very complicated,
and it is not obvious how to use them to deduce Theorem \ref{t.main}.

Theorem \ref{t.main} is equivalent to the assertion that $\pta$ and
$\pcu$ generate the same cone $C$ over $\A$. This in turn is
equivalent to the assertion that $\pta$ and $\pcu$ have the same 
$\A$-indecomposable elements, where an element of a set is 
$\A$-indecomposable if it cannot be written as a positive $\A$-linear
combination of other elements of the set. In the case $\A=\Q$, such
indecomposable elements correspond to the edges of $C$.  More
precisely, each indecomposable element lies on one edge, and each edge
contains one indecomposable element.  

Under certain conditions, Theorem \ref{t.main} can be
strengthened. The $T$-fixed point $x$ can be represented by a Weyl
group element which, by abuse of notation, we denote by $x$ as well.

\begin{theorem}\label{t.ptapcu}
  Suppose that $G$ is simply laced and that (i) $G/P$ is cominuscule;
  or (ii) $x$ is a cominuscule Weyl group element in the sense of
  Peterson; or (iii) $G$ is of type $D$ and $x$ is fully commutative.
  Then $\pta=\pcu$.
\end{theorem}

Part (iii) is a generalization in type $D$ of (ii), since cominuscule
Weyl group elements are fully commutative, but Stembridge (see \cite{Ste:01}) provides
an example in type $D$ of a fully commutative
element which is not cominuscule.  See Remark \ref{r.stem} below.

As an application of Theorem \ref{t.ptapcu}, a smoothness criterion
can be deduced. The Schubert variety $X$ is defined by a Weyl group
element $w\leq x$.  Thus, any reduced expression $\vs$ for $x$
contains a reduced subexpression for $w$.

\begin{theorem}\label{t.smoothcrit}
  Suppose that $G$ is simply laced and any of the three conditions of
  Theorem \ref{t.ptapcu} are satisfied. Let $\vs$ be any reduced
  decomposition for $x$. Then $X$ is smooth at $x$ if and only if
  $\vs$ contains a unique reduced subexpression for $w$.
\end{theorem}

It can be deduced from \cite[Corollary 2.11]{GrKr:15} that this
criterion for smoothness in fact holds whenever $G/P$ is cominuscule,
even without the simply laced requirement. Further discussion of the
criterion appears in \cite{GrKr:23b}.

In addition to the results by Lakshmibai-Seshadri and Carroll-Peterson
discussed above, a number of other papers have studied smoothness and
rational smoothness of Schubert varieties. These include
Lakshmibai-Sandhya \cite{LaSa:90}, Kumar \cite{Kum:96}, Billey
\cite{Bil:98}, Lakshmibai-Littelmann-Magyar \cite{LLM:98}, Brion
\cite{Bri:99}, Billey-Warrington \cite{BiWa:03}, Boe-Graham
\cite{BoGr:03}, Carrell-Kuttler \cite{CaKu:03}, Gaussent
\cite{Gau:03}, and Kassel-Lascoux-Reutenauer \cite{KLR:03}.  We refer
the reader to \cite[Chapters 6 and 8]{BiLa:00} for a detailed
discussion of this topic.

Another application of Theorem \ref{t.main} appears in
\cite{GrKr:23b}, which studies multiplicities of singular points of
Schubert varieties.


\subsection{Outline of the proof of Theorem
  \ref{t.main}}\label{ss.proofoutline}
This theorem builds on the work of \cite{GrKr:20} connecting Demazure
products with weights of tangent spaces.  Here, that connection
appears as Theorem \ref{t.known}, which is proved using the methods of
\cite{GrKr:20}.  The theorem of Carrell and Peterson describing $\pcu$
(\cite{Car:94}) also plays a key role.  The proof of Theorem
\ref{t.main} also requires a detailed study of decomposability of
roots.  An important new ingredient is the notion of
iso-decomposability, which, as shown by Theorem \ref{t.reduced},
appears naturally in the study of inversion sets of Weyl group
elements.

In order to prove Theorem \ref{t.main}, which concerns the Schubert
variety $X$, we first prove an analogous but stronger result for the
Kazhdan-Lusztig variety $Y\st X$ with $T$-fixed point $x$.  The reason
we work with $Y$ rather than $X$ is that the tangent space to $Y$ at
$x$ lives in an ambient space with weights $\pam$, the inversion set
of $x^{-1}$, and algebraic properties related to $\pam$ are essential
to our proof.

Fix a reduced expression $\vs=(s_1, \ldots, s_l)$ for $x$. It is known
that the roots of $\pam$ can be enumerated explicitly by the formula
$\gg_i=s_1\cdots s_{i-1}(\ga_i)$, $i=1,\ldots, l$, where $\ga_i$ is
the simple root corresponding to $s_i$.  Based on this same expression
$\vs$, define $x_i$ and $z_i$ to be the ordinary and Demazure products
respectively of $(s_1,\ldots,\wh{s}_i,\ldots, s_l)$.  Denote the
weights of the tangent space to $Y$ at $x$ by $\ptak$, and the weights
of the tangent lines to $T$-invariant curves of $Y$ through $x$ by
$\pcuk$.  It follows easily from a result of Carrell and Peterson
(\cite{Car:94}) that $\conea \pcuk = \conea \{\gg_i\mid x_i\geq w\}$
(indeed, this result holds before taking cones).  For $\ptak$, we have
the following result, which follows from the methods of
\cite{GrKr:20}.

\begin{theorem}\label{t.known}
  $\ptak \st \conea \{ \gg_i\mid z_i\geq w\}$.
\end{theorem}

See Proposition \ref{p.pcu}(iii).  The main result of \cite{GrKr:20}
is that if $\gg_j$ is an integrally indecomposable element of
$I(x^{-1})$, then $\gg_j$ is in $\ptak$ if and only if $z_i \geq w$.
Theorem \ref{t.known} is an analogue of this result for all elements
of $\ptak$.

The  following equality of cones, which appears in the body of the paper
as Corollary \ref{c.ind-opp}, is our main technical result.

\begin{theorem}\label{t.techmain}
  $\conea \{ \gg_i\mid z_i\geq w\} = \conea \{\gg_i\mid x_i\geq w\}$.
\end{theorem}

The proof of Theorem \ref{t.techmain} is taken up in Sections
\ref{s.decind} - \ref{s.inc-iso}. Before discussing this proof, we
point out that Theorems \ref{t.known} and \ref{t.techmain} together
clearly imply $\ptak\st \conea \pcuk$.  This statement is stronger
than Theorem \ref{t.main}.  Indeed, locally, $X$ differs from $Y$ by a
representation of $T$. Denoting the weights of this representation by
$\gP'$, one can show that $\pta = \gP'\sqcup \ptak$ and $\pcu = \gP'
\sqcup \pcuk$. Thus $\ptak\st \conea \pcuk$ implies $\pta\st
\conea\pcu$.

The proof of Theorem \ref{t.techmain} entails establishing various
relationships among cones, indecomposability, 0-Hecke algebras, and
Weyl groups.  The inclusion $\conea \{ \gg_i\mid x_i\geq w\} \st
\conea \{\gg_i\mid z_i\geq w\}$ follows from the fact that $x_i\leq
z_i$ for all $i$. In order to prove the other inclusion, we introduce
two types of indecomposable elements: {\it increasing
 $\A$-indecomposable} and {\it iso-indecomposable}. Their definitions are
deferred to Section \ref{s.indec}.  In Sections \ref{s.decind},
\ref{s.inc-iso}, and \ref{s.inversion-set} respectively, we prove:
\begin{enumerate}
\item Every element of $\{\gg_i\mid z_i\geq w\}$ is a positive
  $\A$-linear combination of increasing $\A$-indecomposable elements
  which lie in $\{\gg_i\mid z_i\geq w\}$ (Corollary \ref{c.inc-comb}).
\item Increasing $\A$-indecomposable elements are iso-indecomposable
  (Proposition \ref{p.invc}).
\item Iso-indecomposable elements $\gg_i$ satisfy $z_i=x_i$ (Corollary
  \ref{c.ind-zixi}).
\end{enumerate}
Together these statements imply that every element of
$\{\gg_i\mid z_i\geq w\}$ is a positive $\A$-linear combination of
elements which lie in $\{\gg_i\mid x_i\geq w\}$. Thus
$\{ \gg_i\mid z_i\geq w\} \st \conea \{\gg_i\mid x_i\geq w\}$,
completing the proof of Theorem \ref{t.techmain}.

The proof of Theorem \ref{t.main} does not rely on Sections
\ref{s.dec-iso-dec} and \ref{s.allin}.  The main result of these
sections, Theorem \ref{t.iso-gen}, is that for inversion sets in
classical root systems, the notions of rational indecomposability and
iso-indecomposability are equivalent, and in simply laced types are
equivalent to integral indecomposability.  This result, which we view
as of independent interest, has a number of consequences, and is used
in the proofs of Theorems \ref{t.ptapcu} and \ref{t.smoothcrit}.

\subsection{Organization of the paper}\label{ss.org}

In Section \ref{s.indec} we define various notions of decomposition
and indecomposability in {\it root sets}, where a root set is a
generalization of the set of positive roots of a root system. Sections
\ref{s.decind} - \ref{s.allin} mainly address general root sets and
the root set $\pam$.  The purpose of Sections \ref{s.decind},
\ref{s.inversion-set}, and \ref{s.inc-iso}, as discussed above, is to
prove Theorem \ref{t.techmain}, the main technical result needed for
the proof of Theorem \ref{t.main}. The purpose of Sections
\ref{s.dec-iso-dec} and \ref{s.allin} is to establish further
properties of indecomposability in $\pam$ needed in later sections to
prove Theorems \ref{t.ptapcu} and \ref{t.smoothcrit}
respectively. Specifically, in Section \ref{s.dec-iso-dec}, we show
that various types of indecomposability in $\pam$ are equivalent in
classical types; in Section \ref{s.allin}, we examine conditions under
which all elements of a root set are indecomposable.

In Sections \ref{s.schubert} - \ref{s.smoothness}, we narrow our focus
to the root sets $\pcu, \pta\st \pam$. In Section \ref{s.schubert}, we
review known properties of these two roots sets together with some
known facts about Schubert varieties, Kazhdan-Lusztig varieties, and
$T$-invariant curves. In Section \ref{s.ind-tinv}, Theorem
\ref{t.main} is proved, and in Section \ref{s.smoothness}, Theorems
\ref{t.ptapcu} and \ref{t.smoothcrit} are proved.  Finally, Section \ref{s.examples}
contains examples:
we apply Theorem \ref{t.main} to study tangent spaces of singular three-dimensional Schubert varieties,
and verify Theorem \ref{t.main} by direct calculation for a family of examples in type $D_n$.

\section{Indecomposability in root sets: definitions and basic
  properties }\label{s.indec}

As discussed briefly in Section \ref{s.intro}, there is a connection
between cones and indecomposability.  This connection is explored in
Section \ref{s.decind}. In this section we focus on
indecomposability. We introduce various types of indecomposability and
prove some of their basic properties.  In order to study
indecomposability in a general framework, we introduce the notion of a
{\it root set}.

\subsection{Indecomposability definitions}\label{ss.indecdefs}

Let $M$ be a lattice which is isomorphic to $\Z^n$.  Let $\vrr$ be the
associated real vector space $M\otimes_{\Z} \R$ and $\vqq$ the
rational subspace $M\otimes_{\Z} \Q$.  Recall that a subset $S$ of
$\vrr$ is contained in an open half-space of $\vrr$ if there is a
positive definite inner product $\IP{\,}{\,}$ on $\vrr$ and an element
$\delta$ in $\vrr$ such that $\IP{\delta}{\alpha}>0$ for all $\alpha\in
S$. We define a \textbf{root set} to be a finite subset $S$ of $M$
such that $S$ is contained in an open half-space of $\vrr$ and does
not contain both $\alpha$ and $c\alpha$ for a scalar $c\neq 1$. An
element of a root set is often referred to as a \textbf{root}.  A
\textbf{weight} on a root set is a map $z\colon S\to W$, where $W$ is
a partially ordered set.  If $W$ consists of a single element and $z$
is the constant map, then the weight is said to be \textbf{trivial}.
A root set with a weight is called a \textbf{weighted root set}.  Any
subset of a root set is a root set, and any subset of a weighted root
set is a weighted root set.  In our applications, $M$ will be the root
lattice of a semisimple Lie algebra, $S$ a set of positive roots, and
$W$ the Weyl group.

\begin{definition}\label{d.legion}
  Let $S$ be a root set. Let $\A \st \R$ and $I\st S$.  A linear
  combination $\ga=\sum c_i\ga_i$, with $c_i\in \R_{\geq 0}$ and
  $\ga,\ga_i\in S$, is said to be
  \begin{itemize}
  \item an $\mathbf{\A}$\textbf{-linear combination} if $c_i\in \A$
    for all $i$,
  \item \textbf{in} $\mathbf{I}$ or \textbf{by elements of}
    $\mathbf{I}$ if $\ga_i\in I$ for all $i$,
  \item a \textbf{decomposition} if $\ga_i\neq \ga$ for all $i$,
  \item an $\mathbf{\A}$\textbf{-decomposition} if $c_i\in \A$ and
    $\ga_i\neq \ga$ for all $i$,
  \item an \textbf{iso-decomposition} if it is a $\Q$-decomposition of
    the form $\ga=c\ga_1+c\ga_2$ with $\|\ga_1\|=\|\ga_2\|$,
  \item \textbf{increasing} if $S$ is weighted and $z(\ga_i)\geq
    z(\ga)$ for all $i$.
  \end{itemize}
\end{definition}

Elements of $\rs$ for which there exists no decomposition are said to
be \textbf{indecomposable}.  Elements $\ga\in S$ for which there
exists no $A$-decomposition (resp. iso-decomposition, increasing $\A$
decomposition) are said to be \textbf{$\A$-indecomposable} (resp.
\textbf{iso-indecomposable}, \textbf{increasing $\A$-indecomposable});
the set of all such $\ga$ is denoted by $S\ua$ (resp.  $S\uis$,
$S\uia$).

When referring to $\A$-linear combinations, $\A$-decompositions, or
$\A$-indecomposability, the term \textbf{rational} or
\textbf{integral} is often substituted for $\A$ when $\A=\Q$ or
$\A=\Z$ respectively.

\begin{figure}[tbh]
  \begin{equation*}
    \begin{tikzpicture}[scale=0.6]
      \matrix (m) [matrix of math nodes, row sep = 3em, column sep =
      2.5 em] %
      {& S\uiq &\\
        S\uq & S\uis & S\uiz\\
        & S\uz & \\};%
      \path[line width=1.2pt,>=stealth,->] %
      (m-1-2) edge (m-2-3) %
      (m-2-1) edge (m-2-2) %
      (m-2-1) edge (m-3-2) %
      (m-2-1) edge (m-1-2) %
      (m-3-2) edge (m-2-3); %
    \end{tikzpicture}
  \end{equation*}
  \caption{Relationships among five types of indecomposability. Arrows
    represent inclusions.}
\label{f.indec}
\end{figure}
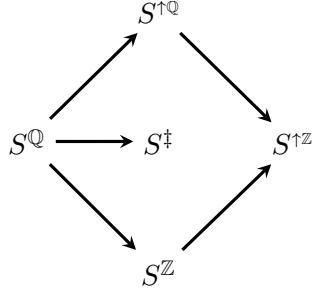

Let $S$ be a root set and $\A\st \R$. Define
\begin{equation*}
  \conea S = \bigg\{ \sum c_i\ga_i : c_i \in \A_{\geq 0},  \alpha_i\in S\bigg\}
  \subseteq V_{\R}.
\end{equation*}
The set $S$ is said to generate $\conea S$.

\subsection{Basic properties of $\A$-indecomposability}\label{ss.bprops}

\begin{lemma}\label{l.indrel}
  Let $E$ and $F$ be  subsets of a root set $S$, and let $A\st
  \R$. Then
  \begin{enumerate}
  \item $F\cap S\ua \subseteq F\ua$.
  \item $F\cap S\ua = F\ua \cap S\ua $.
  \item If $F\ua = E\ua$, then $F\cap S\ua = E\cap S\ua$.
  \item $F\cap S\ua=(F\cap S\ua)\ua$.
  \item If $F^A=F$ and $E\st F$, then $E^A=E$.
  \end{enumerate}
\end{lemma}
\begin{proof}
  (i) If $\ga\in F$ is $\A$-indecomposable in $S$, then $\ga$ is $\A$-indecomposable in the smaller set $F$.
  
  \noindent (ii) Since $F\ua \st F$, $F\ua \cap S\ua \st F\cap S\ua$.  By
  (i), $F\cap S\ua = (F\cap S\ua) \cap S\ua \st F\ua\cap S\ua$.

  \noindent (iii) By (ii), $F\cap S\ua = F\ua \cap S\ua = E\ua \cap
  S\ua = E \cap S\ua$.

  \noindent (iv) Applying (i) twice, $F\cap S\ua=(F\cap S\ua)\cap
  F\ua\st (F\cap S\ua)\ua$. The other inclusion is clear.

  \noindent (v) By (i), $E=E\cap F=E\cap F^A\st E^A$, and the other
  inclusion is clear.
\end{proof}

\begin{definition}\label{d.sgew}
  If $S$ is a root set with weight $z\colon S\to W$, then we define
  $S\zgew=\{s\in S\mid z(s)\geq w\}$.
\end{definition}

\begin{lemma}\label{l.sam}
  Let $S$ be a root set with weights $z,x\colon S\to W$, and let $\A\st
  \R$. Then
  \begin{enumerate}
  \item $(S\ua)\zgew \st (S\zgew)\ua$.
\item  $(S\ua)\zgew = (S\zgew)\ua \cap S\ua$.
\item If $(S\zgew)\ua = (S_{x\geq w})\ua$, then $(S\ua)\zgew =
  (S\ua)_{x\geq w}$.
\item $(S\ua)\zgew = ((S\ua)\zgew)\ua$.
  \end{enumerate}
\end{lemma}
\begin{proof}
  Noting that $S\zgew \cap S\ua = (S\ua)\zgew$, one obtains (i) - (iv)
  of this lemma from (i) - (iv) respectively of Lemma \ref{l.indrel}
  by setting $E=S_{x\geq w}$ and $F=S\zgew$.
\end{proof}

\begin{remark} \label{r.interest} In this paper, $(S\zgew)\ua$ is of
  more importance than $(S\ua)\zgew$.  The reason is that the elements
  of $(S\zgew)\ua$ are used in decomposing elements of $S\gew$.
  Precisely, in Section \ref{s.decind}, we will show that an element
  of $S\gew$ can be written in terms of indecomposable elements of
  $S\gew$---that is, in terms of elements of $(S\zgew)\ua$.  In
  Corollary \ref{c.summary-class}, we show that if $S$ is an inversion
  set in a root system of classical type, $(S\ua)\zgew = (S\zgew)\ua$.
\end{remark}

\section{Decomposing into indecomposables}\label{s.decind}

The base $\gD\st \gP^+$ consists of the roots $\gg\in\gP^+$ which
cannot be expressed as a sum $\gg=\ga+\gb$, where $\ga,\gb\in
\gP^+$. In the literature, such roots $\gg$ are said to be {\it
  indecomposable}.  (This traditional usage differs from our
definition of indecomposable in Section \ref{s.indec}.)  A fundamental
property of root systems is that every element of $\gP^+$ is a
positive integer linear combination of roots in $\gD$.  Similarly, it
is well-known that a similar property holds for $\rs\ua \st \rs$,
where $\A=\Z$ or $\Q$: every element of $\rs$ is a positive
$\A$-linear combination of roots in $\rs\ua$ (for $\A=\Q$, see Remark
\ref{r.cpc}).

In this section we show that an analogous property
holds for $\rs\uia\st \rs$, where $\rs$ is a weighted root set and
$\A=\Z$ or $\Q$: every element $\ga\in \rs$ is an increasing $\A$-linear 
combination of roots $\ga_i\in \rs\uia$. If $z(\ga)\geq w$,
then, since the linear combination is increasing, $z(\ga_i)\geq w$ for
all $i$. This proves $\rs\zgew \st \conea(\rs\uia \cap \rs\zgew)$, the
main result we will need from this section.

We first study increasing $\Z$-linear combinations and then the more
difficult case of increasing $\Q$-linear combinations. Recall that
$\gd$ is chosen so that $\IP{\gd}{\ga}>0$ for all $\ga\in \rs$.

\begin{proposition}\label{p.dind-iz}
  Let $S$ be a weighted root set. Then every element of $S \setminus
  S\uiz$ has an increasing $\Z$-decomposition by elements of $S\uiz$.
\end{proposition}
\begin{proof}
  Let $\ga\in S\setminus S\uiz$.  We can write $\ga=\sum_i c_i\ga_i$,
  where $\ga_i\in S$ satisfy $z(\ga_i)\geq z(\ga)$, and the $c_i$ are
  positive integers at least two of which are nonzero. For all $i$,
  $\IP{\gd}{\ga_i}<\IP{\gd}{\ga}$. By induction on $\IP{\gd}{\cdot}$, if
  $\ga_i\notin S\uiz$, then $\ga_i$ has an increasing $\Z$
  decomposition by elements of $S\uiz$. We conclude that $\ga$ has an
  increasing $\Z$-decomposition by elements of $S\uiz$.
\end{proof}

The inductive proof above does not extend to the case of increasing
$\Q$-decompositions. This is because $\IP{\gd}{\ga_i}$ may not be strictly
less than $\IP{\gd}{\ga}$. Hence the inductive iteration may not
terminate, as seen in the following example.

\begin{example} \label{ex.nonterminate} Suppose $\Phi$ is of type
  $B_2$, $S = \Phi^+ = \{ \gre_1, \gre_2, \gre_1 - \gre_2, \gre_1 +
  \gre_2 \}$, and $z$ is the trivial weight.  The element $\gre_1$ has a
  $\Q$-decomposition
  \begin{equation} \label{e.difficulty} \gre_1 =
    \frac{1}{2}(\gre_1+\gre_2) + \frac{1}{2}(\gre_1 - \gre_2).
\end{equation}
The long root $\gre_1 + \gre_2$ has a $\Q$-decomposition as a sum
of the short roots $\gre_1$ and $\gre_2$: write this decomposition as
$\gre_1 + \gre_2 = (\gre_1) + (\gre_2)$.  But now we can decompose the
summand $\gre_1$ as in \eqref{e.difficulty}.  Substituting into
\eqref{e.difficulty}, we obtain
$$
\gre_1 = \frac{1}{2} \Big( \frac{1}{2}(\gre_1+\gre_2) +
\frac{1}{2}(\gre_1 - \gre_2)+\gre_2 \Big) + \frac{1}{2}(\gre_1 -
\gre_2).
$$
This process can be repeated indefinitely without terminating. Note
that in this example, $S\uiq=\{\gre_1-\gre_2,\gre_2\}$, and $\gre_1 =
(\gre_1 - \gre_2) + \gre_2$ is the desired $\Q$-decomposition of
$\gre_1$ by elements of $S\uiq$.
\end{example}

Thus, a more complicated approach is required in order to extend
Proposition \ref{p.dind-iz} to increasing $\Q$-decompositions.

\begin{lemma} \label{l.replace} Let $I=\{\alpha_1,\ldots,\alpha_n\}$
  be a subset of a weighted root set $S$. Suppose that
  $\alpha_k=\sum_{i=1}^n c_i \alpha_i$ with $c_i$ nonnegative rational
  numbers, $z(\alpha_i)\geq z(\alpha_k)$ whenever $c_i\neq 0$, and
  $c_j>0$ for some $j\neq k$. Then there exists an increasing
  $\Q$-decomposition $\alpha_k=\sum_{i=1}^n d_i\ga_i$ in $I$ (with
  $d_k=0$).
\end{lemma}

\begin{proof}
We have
\begin{equation} \label{e.replace}
(1 - c_k) \ga_k = \sum_{i \neq k} c_i \ga_i.
\end{equation}
The right side of \eqref{e.replace} has a positive inner product with
$\gd$; hence so does the left side.  This implies that $c_k < 1$.  Set
$d_i = c_i/(1-c_k)$ for $i \neq k$, and $d_k = 0$.  If $d_i \neq 0$,
then $c_i \neq 0$, so $z(\ga_i) \geq z(\ga_k)$.  Hence $\ga_k =
\sum_{i=1}^n d_i \ga_i$ is our desired increasing $\Q$-decomposition.
\end{proof}

\begin{lemma}\label{l.removeroot} Let $S$ be a weighted root set.
  Suppose $I_1=I\cup \{\ga\} \subseteq S$, where $\alpha\notin I$, and
  suppose that $\alpha$ has an increasing $\Q$-decomposition in $I$. Then
  any element of $S$ which has an increasing $\Q$-decomposition in $I_1$
  has an increasing $\Q$-decomposition in $I$.
\end{lemma}
\begin{proof}
  Since $\alpha$ has an increasing $\Q$-decomposition in $I$, we see that
  $|I|\geq 2$. Write $I=\{\alpha_1,\ldots,\alpha_{n-1}\}$, with $n\geq
  3$, and let $\alpha_n=\alpha$. Let $\gg$ be an element of $S$ with
  an increasing $\Q$-decomposition in $I_1$, and let
  \begin{equation}\label{e.removeroot}
    \gg=\sum_{i=1}^n c_i\alpha_i
  \end{equation}
  be such an increasing $\Q$-decomposition.  If $c_n=0$ we are done; thus
  assume $c_n\neq 0$. By hypothesis, there is an increasing
  $\Q$-decomposition of $\alpha_n$ in $I$:
 \begin{equation}\label{e.removeroot2}
   \alpha_n=\sum_{i=1}^{n-1} d_i \alpha_i
 \end{equation}
 where we may assume that $d_j\neq 0$ for some $j$ such that
 $\ga_j\neq \gg$.  Let $e_i=c_i+c_nd_i$ for $i<n$; then $e_i\geq 0$
 and
 \begin{equation}\label{e.st}
   \gamma=\sum_{i=1}^{n-1} e_i\alpha_i.
 \end{equation}

 Since $c_n\neq 0$ and $d_j\neq 0$, $e_j\neq 0$. We claim that if
 $e_i>0$ then $z(\alpha_i)\geq z(\gamma)$. Indeed, if $e_i>0$, then
 either $c_i>0$, in which case the claim follows since
 \eqref{e.removeroot} is a increasing $\Q$-decomposition; or $c_nd_i>0$, in
 which case, since both \eqref{e.removeroot2} and \eqref{e.removeroot}
 are increasing $\Q$-decompositions, we have $z(\ga_i)\geq z(\ga_n)\geq
 z(\gg)$. This proves the claim. If $\gamma\notin I$, then
 \eqref{e.st} is an increasing $\Q$-decomposition of $\gamma$ in
 $I$. Otherwise, apply Lemma \ref{l.replace} to obtain an increasing
 $\Q$-decomposition of $\gamma$ in $I$.
\end{proof}

\begin{lemma} \label{l.indnonempty} If $S$ is a nonempty weighted root
  set, then $S\uiq$ is nonempty.
\end{lemma}

\begin{proof}
  We prove the result by induction on $|S|$.  If $S$ has one element,
  then $S = S\uiq$ and the result holds.  For the inductive step,
  suppose that $S = I \cup \{ \ga \}$, where $|I| \geq 1$ and $\ga
  \not\in I$.  Our inductive hypothesis is that $I\uiq$ is nonempty.
  If $\ga$ is increasing $\Q$ indecomposable in $S$, then $S\uiq$
  contains $\ga$ and we are done, so assume that $\ga$ is increasing
  $\Q$-decomposable in $S$.  We will show that $S\uiq = I\uiq$; this
  suffices.

  Observe that $S\uiq \subseteq I\uiq$.  This holds because any
  element of $I$ which is increasing $\Q$ indecomposable in $S$
  remains increasing $\Q$ indecomposable in the smaller set $I$;
  moreover, $\ga$ is increasing $\Q$-decomposable in $S$. For the
  reverse inclusion $S\uiq\supseteq I\uiq$, we require that if
  $\gamma\in I$ does not have an increasing $\Q$-decomposition in $I$,
  then it does not have an increasing $\Q$-decomposition in $S$. This
  follows from Lemma \ref{l.removeroot}, with $I_1=S$.
\end{proof}

\begin{lemma} \label{l.IJ} Let $S$ be a weighted root set, let $I$ and
  $J$ be disjoint subsets of $S$, and let $I_1 = I \cup J$.  Suppose
  that any $\gb \in J$ has an increasing $\Q$-decomposition in $I_1$.  Then
  any $\gg \in S$ which has an increasing $\Q$-decomposition in $I_1$ has
  an increasing $\Q$-decomposition in $I$.
\end{lemma}

\begin{proof}
  We may assume that $J$ is nonempty, since otherwise the lemma is
  trivial.  Since the set $J$ does not intersect $S\uiq$, we see
  that $I \supset S\uiq$, so by Lemma \ref{l.indnonempty}, $I$ is
  nonempty.  Let $I = \{ \ga_1, \ldots, \ga_r \}$ and $J = \{
  \ga_{r+1}, \ldots, \ga_n \}$, where $r \geq 1$ and $n \geq r+1$.
  Let $C(j)$ be the assertion that $\ga_j$ has an increasing
  $\Q$-decomposition in $\{ \ga_1, \dots, \ga_{j-1} \}$.  We will show that
  $C(j)$ holds for $j \in \{ r+1, \ldots, n \}$.

  The assertion $C(n)$ holds by hypothesis.  Suppose that $r+1\leq
  j<n$ and that $C(j+1), \ldots, C(n)$ hold.  We show by contradiction
  that $C(j)$ holds as well. Assume that it does not. Let
  \begin{equation}\label{e.IJoriginal} 
    \ga_j = \sum_{i=1}^n c_i \ga_i
  \end{equation}
  be an increasing $\Q$-decomposition for $\alpha_j$ in $I_1$.  Among all
  possible $\Q$-decompositions for $\alpha_j$ in $I_1$, assume
  \eqref{e.IJoriginal} is one for which the largest integer $m$ for
  which $c_{m}\neq 0$ is smallest. Since $C(j)$ does not hold, $m\geq
  j+1$. Since $\alpha_j$ has an increasing $\Q$-decomposition in
  $\{\alpha_1,\ldots,\alpha_m\}$ and, by $C(m)$, $\alpha_m$ has an
  increasing $\Q$-decomposition in $\{\ga_{1},\ldots,\ga_{m-1}\}$, Lemma
  \ref{l.removeroot} implies that $\alpha_j$ has an increasing
  $\Q$-decomposition in $\{\alpha_1,\ldots,\alpha_{m-1}\}$. This
  contradicts the minimality of $m$ and proves $C(j)$. By induction,
  $C(j)$ holds for $j\in \{r+1,\ldots,n\}$.

  We now complete the proof of the lemma.  Suppose that $\gg \in S$
  has an increasing $\Q$-decomposition in $I_1 = \{ \ga_1, \ldots, \ga_n
  \}$.  Let $m$ be the smallest integer such that $\gg$ has an
  increasing $\Q$-decomposition in $\{ \ga_1, \ldots, \ga_m \}$.  We must
  show $m \leq r$.  If not, then $m \geq r+1$, so by $C(m)$, $\ga_m$
  has an increasing $\Q$-decomposition in $\{ \ga_1, \ldots, \ga_{m-1} \}$.
  Lemma \ref{l.removeroot} then implies that $\gg$ has an increasing
  $\Q$-decomposition in $\{ \ga_1, \ldots, \ga_{m-1} \}$.  This contradicts
  the minimality of $m$.  We conclude that $m \leq r$, as desired.
\end{proof}

\begin{theorem}\label{t.dind-iq}
  Let $S$ be a weighted root set. Then every element of $S \setminus
  S\uiq$ has an increasing $\Q$-decomposition by elements of $S\uiq$.
\end{theorem}
\begin{proof} 
  Every $\gg\in S\setminus S\uiq$ has an increasing $\Q$-decomposition
  in $S$. Thus, by Lemma \ref{l.IJ} with $I = S\uiq$ and $J = S
  \setminus S\uiq$, every such $\gg$ has an increasing $\Q$
  decomposition in $S\uiq$.
\end{proof}

\begin{corollary}\label{c.inc-comb}
  Let $S$ be a weighted root set, and let $A=\Q$ or $\Z$. Then
  $S\zgew\subseteq \conea ((S\uia)\zgew)$.
\end{corollary}
\begin{proof}
  Let $\ga\in S\zgew$. By Theorem \ref{t.dind-iq} and Proposition
  \ref{p.dind-iz}, $\ga$ is a positive $\A$-linear combination of
  elements $\ga_j\in S\uia$ such that $z(\ga_j)\geq z(\ga)$. Since
  $z(\ga)\geq w$, $z(\ga_j)\geq w$ for all $j$, and thus each $\ga_j$
  lies in $S\zgew$.
\end{proof}

\begin{remark}\label{r.inc-pres-zgw}
  The reason that we introduce inceasing linear combinations in this
  paper is that they preserve $\rs\zgw$, in the sense that if
  $\ga=\sum c_j\ga_j$ is increasing and $\ga\in \rs\zgw$, then
  $\ga_j\in\rs\zgw$ for all $j$. This property is used to prove both
  the above corollary and Lemma \ref{l.iso-preserve}.
\end{remark}

\begin{corollary}\label{c.sinco}
  Let $S$ be a root set, and let $A=\Q$ or $\Z$. Then $S\st \conea
  (S\ua)$, and thus $\conea (S) = \conea (S\ua)$.
\end{corollary}
\begin{proof}
  Take $z$ to be the trivial weight in Corollary \ref{c.inc-comb}.
\end{proof}

\begin{corollary}\label{c.conez-zind}
  Let $E$, $F$ be subsets of a root set $S$. Let $\A=\Q$ or
  $\Z$. Then $\conea(E)=\conea(F)$ if and only if $E\ua=F\ua$.
\end{corollary}
\begin{proof}
  Assume $\conea(E)=\conea(F)$. Suppose that there exists $\ga\in
  E\ua\setminus F\ua$. Then $\ga\in\conea(E)=\conea(F)=\conea(
  F\ua)$. Thus $\ga=\sum_ic_i\ga_i$, where $\ga_i\in F\ua$, and
  $c_i\in A_{>0}$, at least two of which are nonzero (otherwise
  $\ga\in F\ua$). But since $\ga_i\in F\ua \st \conea(E)$, $\ga\notin
  E\ua$, a contradiction. Thus $E\ua\st F\ua$, and similarly $F\ua\st
  E\ua$.

  Conversely, if $E\ua=F\ua$, then
  $\conea(E)=\conea(E\ua)=\conea(F\ua)=\conea(F)$.
\end{proof}

\begin{remark}\label{r.cpc}
  Proofs of Corollaries \ref{c.sinco} and \ref{c.conez-zind} for the
  case $\A=\R$ can be obtained by using the fact that if $S$ is a root
  set, then $\coner S$ is a convex polyhedral cone, and $S\ur$ is
  equal to the set of elements of $S$ which lie on the one-dimensional
  faces of $\coner S$ (see \cite[Section 1.2]{Ful:93}). One can then
  use the density of $\Q$ in $\R$ to obtain alternative proofs of
  these two corollaries for the case $\A=\Q$.
\end{remark}

\section{Iso-indecomposability in inversion
  sets}\label{s.inversion-set}
In this section we introduce the notion of iso-decomposability.  The main results of
this section, Theorem \ref{t.reduced}, and its corollaries, play a major part in this paper.

\subsection{Preliminaries and notation} \label{ss.prelim}
In this section we introduce some notation that will be used throughout the
rest of the paper.
Let $G$ be a semisimple algebraic group defined over an
algebraically closed field of characteristic 0, and let $B\supseteq T$
be a Borel subgroup and maximal torus respectively.  Let $\gP$ be the
set of roots of $G$ relative to $T$.  Let $\gP^+$ and $\gP^-$ be the
sets of positive and negative roots chosen so that $\mbox{Lie}(B)$ is spanned
by positive root spaces.  For the remainder
of this paper we limit attention to root sets $\rs\st \gP$; our
convention will be that any statement involving $\A$ holds for
$\A=\Q$, and if $\gP$ is simply laced, it holds for $\A=\Z$ as well.

Let $W=N_G(T)/T$, the Weyl group of $T$, and let $S'$ be the set of
simple reflections of $W$ relative to $B$.  The longest element of
$W$ is $w_0$.  The 0-Hecke algebra $\ch$
associated to $(W,S')$ over a commutative ring $R$ is the associative
$R$-algebra generated by $H_u$, $u\in W$, and subject to the following
relations: $H_1$ is the identity element, and if $u\in W$ and $s\in
S'$, then $H_uH_s=H_{us}$ if $\ell(us)>\ell(u)$ and $H_{u}H_s=H_u$ if
$\ell(us)<\ell(u)$.

Throughout the paper, we will assume that we have
chosen $w \leq x\in W$ and a reduced expression $\vs=(s_1,\ldots,s_l)$, $s_i\in S'$, for $x$.
Then $x_i$, $z_i$, and $\gg_i$ will have the following meaning.
For $i\in\{1,\ldots, l\}$, define
\begin{itemize}
\item $x_i=s_1\cdots \wh{s}_i\cdots s_l = s_{\gg_i}x \in W$.
\item $z_i\in W$ by the equation $H_{z_i}=H_{s_1}\cdots
  \wh{H}_{s_i}\cdots H_{s_l}$.
\item $\gg_i=s_1\cdots s_{i-1}(\ga_i)\in \gP$, where $\ga_i$ is the
  simple root corresponding to $s_i$.
\end{itemize}
It is known that the elements $\gg_1,\ldots,\gg_l$ enumerate
$\pam=\{\ga\in\gP^+\mid x^{-1}(\ga)\in \gP^-\}$, the inversion set of
$x^{-1}$ (see \cite[Exercise 5.6.1]{Hum:90}).   The equality $x_i=s_{\gg_i}x$
 is well-known; it follows from the equation
 $s_{\gg_i}=(s_1\cdots
   s_{i-1})s_i(s_1\cdots s_{i-1})^{-1}$. 

Let $w\in W$, $w\leq x$. For any root set $\rs\st \pam$, define the
\textbf{Coxeter weight} $x\colon \rs\to W$ by $x(\gg_i)=x_i$, and the
\textbf{Demazure weight} $z\colon \rs\to W$ by $z(\gg_i)=z_i$.  Then
$\rs\xgw =\{\gg_i\in S \mid x_i\geq w\}$ and $\rs\zgw=\{\gg_i\in S\mid
z_i\geq w\}$ (see Definition \ref{d.sgew}).

The main result of Section \ref{s.inversion-set} is that for $\rs=\pam$, $\gg_i$ is
iso-indecomposable in $\rs$ if and only if
$(s_1,\ldots,\wh{s}_i,\ldots,s_l)$ is reduced; in this case,
$z_i=x_i$.
Consequently, $(\rs\uis)\zgw = (\rs\uis)\xgw$.

 \begin{remark}\label{r.redmap}
   Since $x_i=s_{\gg_i}x$, the Coxeter weight $x:\pam\to W$,
   $\gg_i\mapsto x_i$, is independent of the reduced expression $\vs$
   for $x$.  On the other hand, the Demazure weight $z:\pam\to W$,
   $\gg_i\mapsto z_i$, is not.  (For example, let $x=\gs_1\gs_2\gs_1$
   in type $A_2$. For $\vs=(\gs_1,\gs_2,\gs_1)$, one checks that
   $\gg_2=\ga_1+\ga_2$ and $z_2=\gs_1$; for $\vs=(\gs_2,\gs_1,\gs_2)$,
   we again have $\gg_2=\ga_1+\ga_2$, but now $z_2=\gs_2$.)  The
   dependence of the Demazure weight on $\vs$ can be removed by
   restricting the domain to the set of iso-indecomposable elements
   (since $z_i = x_i$ if $\gg_i$ is iso-indecomposable).
 \end{remark}

\subsection{Demazure products}\label{ss.demazure}

If $\vq=(r_1,\ldots,r_k)$ is any (not necessarily reduced) sequence of
simple reflections in $S$, define the Demazure product\footnote{In
  \cite{KnMi:04} and \cite{GrKr:20}, the Demazure product $z_{\vq}$ is
  instead denoted by $\gd(\vq)$.} $z_\vq\in W$ by the equation
$H_{z_\vq}=H_{r_1}\cdots H_{r_k}$, and define $x_{\vq}=r_1\cdots r_k$.
It is well known that if $\vq$ is reduced, then $\vq$ contains a
subexpression which multiplies to $u$ if and only if $x_{\vq}\geq u$
(see Theorem 5.10 of \cite{Hum:90}).  A generalization of this result
in which $\vq$ is not required to be reduced is given by \cite[Lemma
3.4(1)]{KnMi:04}:

 \begin{lemma}\label{l.heckefact}
   $\vq$ contains a subexpression which multiplies to $u$
   $\Leftrightarrow$ $z_\vq \geq u$.
\end{lemma}

\begin{corollary}\label{c.dem-subex}
  There exists a subexpression of $\vq$ which is a reduced expression
  for $z_{\vq}$.
\end{corollary}
\begin{proof}
  By Lemma \ref{l.heckefact} with $u=z_{\vq}$, $\vq$ contains a
  subexpression which multiplies to $z_{\vq}$, and hence it contains a
  reduced subexpression which multiplies to $z_{\vq}$.
\end{proof}

\begin{corollary}\label{c.heckecons}
  We have
  \begin{itemize}
  \item[(i)] $z_\vq\geq x_\vq$, with equality if $\vq$ is reduced.
  \item[(ii)] $z_\vq\geq z_\vp$ if $\vp$ is a subexpression of $\vq$.
  \end{itemize}
\end{corollary}
\begin{proof}
  \noindent (i) If $\vq$ is reduced, then $z_{\vq}=x_{\vq}$ by
  definition. The inequality is due to Lemma \ref{l.heckefact} with
  $u=x_\vq$.

  \noindent (ii) By Corollary \ref{c.dem-subex}, there exists a
  subexpression $\vp'$ of $\vp$ which is a reduced expression for
  $z_\vp$. By (i), $x_{\vp'}=z_\vp$. Since $\vp'$ is a subexpression
  of $\vp$, it is a subexpression of $\vq$, so by Lemma
  \ref{l.heckefact}, $z_\vq\geq x_{\vp'}$. Hence $z_\vq\geq z_\vp$, as
  required.
\end{proof}

\begin{remark}
  In Corollary \ref{c.heckecons}(i), equality can occur even if $\vq$
  is not reduced. For example, if $\gs_1$ and $\gs_2$ denote the
  transpositions $(1,2)$ and $(2,3)$ respectively in type $A_2$, then
  for $\vq=(\gs_1,\gs_1,\gs_2,\gs_1,\gs_2)$,
  $x_\vq=z_\vq=\gs_1\gs_2\gs_1$, although $\vq$ is not reduced.
\end{remark}

\begin{corollary}\label{c.zixi}
  $z_i\geq x_i$ for $i\in \{1,\ldots,l\}$, with equality if
  $(s_1,\ldots,\wh{s}_i,\ldots,s_l)$ is reduced.
\end{corollary}
\begin{proof}
  This is a special case of Corollary \ref{c.heckecons}(i).
\end{proof}

\subsection{Iso-indecomposability in $I(x^{-1})$.}\label{ss.iso-indec}

\begin{lemma}\label{l.d-decomp}
  Let $\ga,\gb,\gg\in \Phi^+$. If $c\ga = \gb+\gg$ and $\| \gb\|=
  \|\gg\|$, then $c=\IP{\gb}{\ga^{\vee}}=\IP{\gg}{\ga^{\vee}}>0$.
\end{lemma}
\begin{proof}
  Since $\ga,\gb,\gg>0$, we must have $c>0$.  By applying
  $\IP{\cdot}{\ga^{\vee}}$ to both sides of the equation
  $c\ga=\gb+\gg$, we find that
  $c=(1/2)\IP{\gb}{\ga^{\vee}}+(1/2)\IP{\gg}{\ga^{\vee}}$.  Since
  $\|\gb\|=\|\gg\|$,
  \begin{equation*}
    (\gb,c\ga)=(\gb,\gb+\gg)=\| \gb\|^2+(\gb,\gg)=\|
    \gg\|^2+(\gb,\gg)=(\gg, \gb+\gg)=(\gg,c\ga),
  \end{equation*}
  implying $\IP{\gb}{\ga^{\vee}}=\IP{\gg}{\ga^{\vee}}$, as desired.
\end{proof}

\begin{lemma}\label{l.simpiso}
  Let $\ga,\gb,\gg\in \gP^+$, where $\gP$ is simply laced.  If
  $c\ga=\gb+\gg$ and $\gb\neq \ga$, then $c=1$.
\end{lemma}
\begin{proof}
  This follows from Lemma \ref{l.d-decomp}.
\end{proof}

\begin{proposition}\label{p.int-iso}
  If $S\subseteq \gP^+$ is a root set and $\gP$ is simply laced, then
  $\zind{S}\subseteq \isnd{S}$.
\end{proposition}
\begin{proof}
  Suppose $\ga$ is iso-decomposable in $S$. Then $c\ga = \gb+\gg$, for
  some $\gb,\gg\in S$. By Lemma \ref{l.simpiso}, $c=1$. Thus $\ga$ is
  integrally decomposable in $S$.
\end{proof}

\begin{lemma}\label{l.red_in} Let $i,j,k\in [l]$.
  \begin{enumerate}
  \item $\gg_i\neq \gg_j$ if $i\neq j$.
  \item If $j<i$ then $s_{j}\cdots s_1 (\gg_i)>0$; otherwise
    $s_{j}\cdots s_1 (\gg_i)<0$.
  \item If $c\gg_i=\gg_j+\gg_k$, $j<k$, then $j<i<k$.
  \end{enumerate}
\end{lemma}
\begin{proof}
  (i), (ii)  See \cite[Section 1.7]{Hum:90}.

  \noindent (iii) By Lemma \ref{l.d-decomp}, $c>0$. Assume that $i<j$
  and $i<k$, and let $y=s_i\cdots s_1$. Then
  $cy\gg_i=y\gg_j+y\gg_k$. By (ii), $cy\gg_i<0$ and $y\gg_j+y\gg_k>0$,
  contradiction. A similar argument eliminates the possibility that
  $i>j$ and $i>k$. (This proof also appears in the proof of
  \cite[Theorem 5.3]{Ste:01}.)
\end{proof}

The following theorem plays a central role in this paper: it is
required for all subsequent results of this section and the next.
The theorem also illustrates how iso-indecomposability arises
naturally in the study of Coxeter groups and inversion sets.  Indeed,
the problem of finding some property of $\gg_i$ which is equivalent to
reducedness of $s_1\cdots \wh{s}_i\cdots s_l$ initially led us to this
theorem and the definition of iso-indecomposability.
\begin{theorem}\label{t.reduced}
  For $i\in [l]$, the following are equivalent:
  \begin{itemize}
  \item[(i)] $s_1\cdots \wh{s}_i\cdots s_l$ is not reduced.
  \item[(ii)] There exist $j<i<k$ such that $\ga_j=s_{j+1}\cdots
    \wh{s}_i\cdots s_{k-1}(\ga_k)$.
  \item[(iii)] There exist $j<k$ and $c\in \Q$ such that
    $c\gg_i=\gg_j+\gg_k$ and $\|\gg_j\| = \|\gg_k\|$.
  \item[(iv)] $\gg_i$ is iso-decomposable in $I(x^{-1})$.
  \end{itemize}
  Moreover, in case one (and thus all) of these statements hold, it
  must be true that $j<i<k$; and that $j,k$ satisfy (ii) if and only
  if they satisfy $(iii)$; and that $c=\IP{\gg_k}{\gg_i^{\vee}}>0$.
  \end{theorem}
  \begin{proof}
  (i) $\Leftrightarrow$ (ii) See \cite[Theorem 1.7]{Hum:90}; (iii)
  $\Leftrightarrow$ (iv) by definition.

  Now suppose $j<i<k$. Let $\gb=s_{i+1}\cdots s_{k-1}(\ga_k)$.
  Then $\IP{\gg_k}{\gg_i^{\vee}}=\IP{s_1\cdots s_i\gb}{-s_1\cdots
    s_i\ga_i^{\vee}}=\IP{-\gb}{\ga_i^{\vee}}$. Thus
  \begin{equation}\label{e.gkgi}
    \begin{split}
      \IP{\gg_k}{\gg_i^{\vee}}\gg_i &= \IP{-\gb}{\ga_i^{\vee}}\gg_i\\
      &=  s_1\cdots s_{i-1}(\IP{-\gb}{\ga_i^{\vee}}\ga_i)\\
      &= s_1\cdots s_{i-1}(-\gb+s_i\gb) = - s_1\cdots 
      \wh{s}_i\cdots s_{k-1}(\ga_k)+\gg_k.
    \end{split}
  \end{equation}

  \noindent (ii) $\Rightarrow$ (iii) Let $j<i<k$ be as in
  (ii). Substituting $\ga_j$ for $s_{j+1}\cdots\wh{s}_i\cdots
  s_{k-1}(\ga_k)$ in \eqref{e.gkgi} produces
  $\IP{\gg_k}{\gg_i^{\vee}}\gg_i=\gg_j+\gg_k$.

  \noindent (iii) $\Rightarrow$ (ii) Lemma \ref{l.red_in}(iii) forces
  $j<i<k$. By Lemma \ref{l.d-decomp},
  $c=\IP{\gg_k}{\gg_i^{\vee}}$, so
 $\IP{\gg_k}{\gg_i^{\vee}} \gg_i=\gg_j+\gg_k$.
  Substituting this into  \eqref{e.gkgi}, we
  obtain $\gg_j = - s_1\cdots  \wh{s}_i\cdots
    s_{k-1}(\ga_k)$.  On the other hand, by definition,
    $\gg_j = s_1\cdots s_{j-1}(\ga_j)$.  Equating these two
    expressions for $\gg_j$ and simplifying yields
$s_{j+1}\cdots \wh{s}_i\cdots s_{k-1}(\ga_k)=\ga_j$,
  as required.
 \end{proof}

 \begin{corollary}\label{c.ind-zixi}
   Let $\rs=\pam$. If $\gg_i\in \rs\uis$, then $z_i=x_i$.
 \end{corollary}
 \begin{proof}
   If $\gg_i\in\rs\uis$, then $(s_1,\ldots,\wh{s}_i,\ldots,s_l)$ is
   reduced by Theorem \ref{t.reduced}, and thus $z_i=x_i$ by Corollary
   \ref{c.zixi}.
 \end{proof}

\begin{corollary}\label{c.inds}
  Let $\rs=\pam$. Then $(\rs\uis)\zgw = (\rs\uis)\xgw$ and $(\rs\ua)\zgw =
  (\rs\ua)\xgw$.
\end{corollary}
\begin{proof}
  The first equation is due to Corollary \ref{c.ind-zixi}.  
Note that the first equation implies that for any $E \subseteq \rs\uis$,
  we have $E\zgw = E\xgw$.  The second
  equation now follows by observing that $\rs\uq \st
  \rs\uis$, and if $\gP$ is simply laced, then $\rs\uz\st \rs\uis$ by
  Proposition \ref{p.int-iso}.
\end{proof}

\begin{remark}\label{r.prev-paper}
  Corollary \ref{c.inds} proves the assertions of \cite[Remark
  5.9]{GrKr:20}.
\end{remark}

\section{Increasing and iso-indecomposability}\label{s.inc-iso}

In this section we show that for $\rs=\pam$ we have that
$\rs\uia \st \rs\uis$, where increasing $\A$-decompositions are
relative to the Demazure weight.  Using this and results of previous
sections, we prove our main technical result:
$\conea (S\zgw) = \conea (S\xgw)$.

\begin{lemma}\label{l.iso-inc}
  Let $\rs=\pam$, and let $\gg\in \rs$. If there exists an
  iso-decomposition of $\gg$ in $\rs$, then there exists an increasing
  iso-decomposition of $\gg$ in $\rs$.
\end{lemma}
\begin{proof}
  We have that $\gg=\gg_i$ for some $i$.  Assume that $\gg_i$ is
  iso-decomposable.  By Theorem \ref{t.reduced},
  $s_1\cdots \wh{s}_i\cdots s_l$ is not reduced.  Choose $k>i$ minimal
  such that
  $\ell(s_1\cdots \wh{s}_i\cdots s_k)<\ell(s_1\cdots \wh{s}_i\cdots
  s_{k-1})$, and then $j<i$ maximal such that
  $\ell(s_j\cdots \wh{s}_i\cdots s_{k})<\ell(s_{j+1}\cdots
  \wh{s}_i\cdots s_{k})$. It is shown in the proof of \cite[Theorem
  1.7]{Hum:90} that $j,k$ satisfy Theorem \ref{t.reduced}(ii); thus
  they satisfy Theorem \ref{t.reduced}(iii), i.e.,
  $c\gg_i=\gg_j+\gg_k$ with $\|\gg_j\|=\|\gg_k\|$. We will show that
  $z_j,z_k\geq z_i$, thus completing the proof.

  By choice of $k$, $s_1\cdots \wh{s}_i\cdots s_{k-1}$ is reduced but
  $s_1\cdots \wh{s}_i\cdots s_k$ is not.  Thus
  \begin{equation*}
    H_{s_1}\cdots \wh{H}_{s_i}\cdots H_{s_{k-1}}=H_{s_1\cdots
      \wh{s}_i\cdots s_{k-1}}=H_{s_1\cdots
      \wh{s}_i\cdots s_{k-1}}H_{s_k} = H_{s_1}\cdots \wh{H}_{s_i}\cdots H_{s_k}
  \end{equation*}
  If we multiply on the right by $H_{s_{k+1}}\cdots H_{s_l}$, we
  obtain $H_{s_1}\cdots \wh{H}_{s_i}\cdots \wh{H}_{s_k}\cdots
  H_{s_l}=$ $H_{s_1}\cdots \wh{H}_{s_i}\cdots H_{s_l}$.  Letting
  $\vr=(s_1,\ldots,\wh{s}_i,\ldots,\wh{s}_k,\ldots,s_l)$, we have
  $z_{\vr}=z_i$. Since $\vr$ is a subexpression of
  $(s_1,\ldots,\wh{s}_k,\ldots,s_l)$, Corollary \ref{c.heckecons}(ii)
  implies $z_k\geq z_{\vr}$. Therefore $z_k\geq z_i$.

  Using the fact that $s_{j+1}\cdots \wh{s}_i\cdots s_{k}$ is reduced
  but $s_j\cdots \wh{s}_i\cdots s_{k}$ is not, a similar argument
  yields $z_j\geq z_i$.
\end{proof}

\begin{lemma}\label{l.iso-preserve}
  Let $\rs=\pam$, and let $\gg\in \rs\gew$. If there exists an
  iso-decomposition of $\gg$ in $\rs$, then there exists an
  iso-decomposition of $\gg$ in $\rs\gew$.
\end{lemma}
\begin{proof}
  This follows from Lemma \ref{l.iso-inc} and the fact that increasing
  linear combinations preserve $\rs\gew$ (see Remark
  \ref{r.inc-pres-zgw}).
\end{proof}

\begin{proposition}\label{p.invc}
  Let $\rs=\pam$. Then $\rs\uia\st \rs\uis$.
\end{proposition}
\begin{proof}
  Suppose $\gg_i\notin\rs\uis$.  By Lemma \ref{l.iso-inc}, there exist
  $j,k$ such that $c\gg_i=\gg_j+\gg_k$ is increasing and $c\in\Q$.  If
  $\gP$ is simply laced, then $c=1$, by Lemma \ref{l.d-decomp}. Thus
  $\gg_i\notin \rs\uia$.
\end{proof}

\begin{corollary}\label{c.isoind}
  Let $\rs=\pam$. Then $\isnd{(S\zgw)} = (\isnd{S})\zgw$.
\end{corollary}
\begin{proof}
  By definition, $ (\isnd{S})\zgw = \isnd{S} \cap S\zgw$.  The
  inclusion $\isnd{(S\zgw)} \supseteq (\isnd{S})\zgw$ holds because
  any element of $S\zgw$ that is iso-indecomposable in $S$ is
  iso-indecomposable in the smaller set $S\zgw$.  The reverse
  inclusion is the contrapositive of Lemma \ref{l.iso-preserve}.
\end{proof}

\begin{corollary}\label{c.ind-opp}
  Let $\rs=\pam$. Then $\conea (\rs\zgw) = \conea (\rs \xgw)$, or
  equivalently, $(\rs\zgw)\ua=(\rs\xgw)\ua$.
\end{corollary}
\begin{proof}
  Since $x_i\leq z_i$ for all $i$, $\conea (\rs\xgw) \st \conea (\rs
  \zgw)$. The other inclusion follows from $ \rs\zgw \st \conea
  ((\rs\uia)\zgw) \st \conea ((\rs\uis)\zgw) = \conea
  ((\isnd{\rs})\xgw) \st \conea (\rs\xgw)$, where the first and second
  inclusions are due to Propositions \ref{c.inc-comb} and
  \ref{p.invc}, and the equality is due to Corollary \ref{c.inds}.
  The equivalence of the second equality of this corollary is due to
  Corollary \ref{c.conez-zind}.
\end{proof}

\begin{remark} \label{r.ind-opp} By Lemma \ref{l.sam}(iii), Corollary
  \ref{c.ind-opp} is a stronger statement than
  $(\rs\ua)\zgw = (\rs\ua)\xgw$, which was proved in Corollary
  \ref{c.inds}. 
\end{remark}

\section{Equivalent indecomposabilities}\label{s.dec-iso-dec}
The main theorem of this section is the following.

\begin{theorem} \label{t.iso-gen} Let $\Phi$ be of classical type, and let $S = I(x^{-1})$.
Let $\ga \in S$.  Then
$\ga$ is rationally indecomposable $\Leftrightarrow$ $\ga$ is iso-indecomposable.
If $\Phi$ is simply laced, these conditions are equivalent to the condition
that $\ga$ is integrally indecomposable.
\end{theorem}

A more precise statement, Proposition \ref{p.iso-bi-gen}, is given in Section \ref{ss.indclos}.
The special case of $S = \Phi^+$ is studied in Section \ref{ss.indgp}.

The equivalence of these indecomposabilities, which we view as of
independent interest, has two main applications in this paper.  First,
it is used in this section to prove that in classical types,
$(\rs\ua)\zgw=(\rs\zgw)\ua$ and $(\rs\ua)\xgw=(\rs\xgw)\ua$. In
Section \ref{s.ind-tinv}, we show that the first of these equalities
implies that in classical types, $\ptak\cap \rs\ua =
(\ptak)\ua$. Second, it allows us to prove Corollary \ref{c.adind}: in
types $A$ and $D$, $x$ is fully commutative if and only if all
elements of $\rs$ are integrally indecomposable. This leads, in
Section \ref{s.smoothness}, to smoothness criteria for fully
commutative $x$ in types $A$ and $D$.

\subsection{Indecomposability in closed subsets of $\gP^+$} \label{ss.indclos}

It is convenient to introduce the following characterization of inversion sets,
which is a slight variation of the characterization given by Papi \cite{Pap:94}.
We shall say that a root set $S\st \gP^+$ is {\it closed} if (i) $\ga,
\gb \in S$ and $r \ga+s\gb \in \Phi$ for positive real numbers $r$ and
$s$ imply $r\ga+s\gb \in S$, and (ii) $\ga,\gb\in \Phi^+$ and
$\ga+\gb\in S$ imply $\ga\in S$ or $\gb\in S$.  
The following lemma is a simple consequence of Papi's characterization.

\begin{lemma} \label{l.closed}
Let $S \st \gP^+$.  Then $S$ is closed $\Leftrightarrow$ $S = \pam$ for some $x \in W$.
\end{lemma}

\begin{proof}
Consider the condition ($\mbox{i}'$): $\ga,
\gb \in S$ and $\ga+s\gb \in \Phi$ 
$s$ implies $\ga+\gb \in S$.  Papi proved that $S$ satisfies ($\mbox{i}'$) and (ii) if and only
if $S$ is of the form $\pam$.

Suppose $S$ is closed.  Since $S$ satisfies (i) and (ii),
it satisfies ($\mbox{i}'$) and (ii), so $S = \pam$ for some $x \in W$.
Conversely, suppose $S = \pam$.  The set $S$ satisfies (ii) by Papi's result.  If $\ga,
\gb \in S$ and $r \ga+s\gb \in \Phi$ for positive real numbers $r$ and
$s$, then $x^{-1} \ga$ and $x^{-1} \gb$ are negative roots, so the root $x^{-1}(r \ga + s \gb)$ must
be negative as well (because when written as a sum of simple roots, all coefficients of
$x^{-1}(r \ga + s \gb)$ are negative).  Hence $r \ga+s\gb \in S$, so $S$ satisfies (i).  Hence $S$ is closed.
\end{proof}

Theorem \ref{t.iso-gen} is an immediate consequence of Proposition
\ref{p.iso-bi-gen} below.  We divide the proof into smaller steps by
introducing a new variation of indecomposability.

\begin{definition}\label{d.bdecomp}
  Let $S$ be a root set. A rational decomposition in $S$ of the form
  $\ga= c_1 \ga_1 + c_2 \ga_2$ is called a \textbf{bi-decomposition}
  of $\ga$. The element $\alpha$ is called \textbf{bi-decomposable} if
  $\alpha$ has a bi-decomposition and \textbf{bi-indecomposable}
  otherwise.
\end{definition}

\begin{proposition}\label{p.iso-bi-gen}
  Let $S$ be a closed subset of $\Phi^+$ and let $\alpha\in
  S$. Consider the following conditions:
  \begin{enumerate}
  \item  $\alpha$ is iso-decomposable.
  \item $\alpha$ is bi-decomposable.
  \item $\alpha$ is rationally decomposable.
  \item $\ga$ is integrally decomposable.
  \end{enumerate}
  We have (i) $\Leftrightarrow$ (ii) $\Rightarrow$ (iii) $\Leftarrow$
  (iv). If $\Phi$ is of classical type, then (iii) $\Rightarrow$ (ii)
  and thus (i) -- (iii) are equivalent. If $\gP$ is simply laced, then
  (i) $\Rightarrow$ (iv). Thus in types $A$ and $D$, (i) - (iv) are
  equivalent.
\end{proposition}
\begin{proof}
  (i) $\Rightarrow$ (ii) $\Rightarrow$ (iii) $\Leftarrow$ (iv) is
  clear from the definitions; (i) $\Rightarrow$ (iv) if $\Phi$ is
  simply laced, by Proposition \ref{p.int-iso}.

  \noindent (ii) $\Rightarrow$ (i) We first provide several facts
  about rank 2 root systems, referring the reader to \cite[Ch. 5, \S 3
  and \S 7]{Ser:01} for additional background and details.  There are
  four rank 2 root systems: $A_1\times A_1$, $A_2$, $B_2$, and
  $G_2$. Every root $\lambda$ of a rank 2 root system has two
  ``nearest neighbors'', one to either side, which we denote by
  $\lambda'$ and $\lambda''$. One checks that $|\gl'|=|\gl''|$, and in
  types $A_2$, $B_2$, and $G_2$, $\gl = c \gl' + c \gl''$, where $c=1$
  or $c=1/2$. This gives an iso-decomposition of $\gl$ in $\Phi$.

Since $\alpha$ is bi-decomposable, we can write
\begin{equation}\label{e.bdecomp}
  \ga = r \gb + t \gg
\end{equation}
where $r,t$ are positive rational numbers, $\gb,\gg\in S$, and
$\ga,\gb,\gg$ are distinct. Let $X=\{\ga,\gb,\gg \}$. Since the three
elements of $X$ are distinct, they do not lie on the same line through
0. By \eqref{e.bdecomp}, they lie on the same plane through 0. Thus
the $\R$ span of $X$, which we denote by $V_X$, is two dimensional. By
\cite[Ch. VI, \S 1, no. 1, Prop. 4(ii)]{Bou:02}, $\Phi_X:= \Phi \cap
V_X$ is a root system. Its rank is two, and thus it must be of type
$A_1\times A_1$, $A_2$, $B_2$, or $G_2$. Since it contains $X$, and
$X$ contains three elements none of which is the negative of any
other, we can rule out type $A_1\times A_1$.

Let $C=\coner \{\gb,\gg\}=\{x\gb + y\gg: x,y\in \R, x,y\geq 0\}$, the
convex polyhedral cone generated by $\{\gb,\gg\}$. Geometrically, $C$
is the locus of points in $V_X$ lying on or between the ray through
$\beta$ and the ray through $\gamma$. By \eqref{e.bdecomp}, $\alpha$
is in the interior of $C$. Since $\gb,\gg \in \Phi_X$, it follows that
$\alpha', \alpha''$, the nearest neighbors to $\alpha$ in $\Phi_X$,
must also lie in $C$. Since $\gb,\gg \in S$ and $\ga', \ga'' \in
\Phi$, condition (i) of $S$ being closed implies $\ga',\ga'' \in
S$. As shown above, $\ga = c\ga' + c \ga''$, where $c=1$ or
$c=1/2$. Hence $\ga$ is iso-decomposable in $S$.

\noindent (iii) $\Rightarrow$ (ii) if $\Phi$ is of classical type. We
proceed by contradiction.  Suppose that $\alpha$ is rationally
decomposable but not bi-decomposable.  Then we can write
$\ga=r_1\ga_1+\cdots+r_n\ga_n$, where $n\geq 3$, $r_i$ positive
rational numbers, $\ga_i \in \gP^+$ distinct, and $\ga_i \neq \ga $,
but we cannot express $\ga$ as such a linear combination with
$n=2$. Clearing denominators, we obtain
  \begin{equation}\label{e.zcomb}
    d\ga = m_1\ga_1 + \cdots + m_n\ga_n,
  \end{equation}
  where $d,m_1,\ldots,m_n$ are positive integers.  Among such
  expressions, consider those with $d$ minimal; among these, consider
  those with $m_1+\cdots+m_n$ minimal; among these, choose one with
  $m_1|\ga_1|^2+\cdots+m_n|\ga_n|^2$ minimal. Assume that $d$, $n$,
  $m_1,\ldots,m_n$, $\ga_1,\ldots,\ga_n$ are so chosen.  We can
  rewrite \eqref{e.zcomb} as
  \begin{equation}\label{e.zsum}
    \ga+\cdots + \ga = (\ga_1+\cdots + \ga_1) + \cdots + (\ga_n+ \cdots
    + \ga_n),
  \end{equation}
  where $\ga$ occurs $d$ times and each $\ga_i$ occurs $m_i$
  times. The total number of summands on the right hand side is
  $m_1+\cdots + m_n$.

  We make several preliminary observations about \eqref{e.zsum}: for
  all $i,j$ such that $i\neq j$,
  \begin{itemize}
  \item[(a)] $\ga_i+\ga_j$ is not an integer multiple of $\ga$.
  \item[(b)] $\ga_i+\ga_j\notin \Phi^+$.
  \item[(c)] $(\ga_i,\ga_j)\geq 0$.
  \item[(d)] $(\ga,\ga_i)>0$.
  \end{itemize}
  To prove (a), note that if $\ga_i+\ga_j=e\ga$ for some positive
  integer $e$, then $\alpha$ is bi-decomposable in $S$, a
  contradiction.  To prove (b), suppose that $\ga_i+\ga_j\in \Phi^+$.
  Then, since $S$ is closed, $\ga_i+\ga_j\in S$. Thus two summands
  $\ga_i,\ga_j$ on the right hand side of \eqref{e.zsum} can be
  replaced by the single summand $\ga_i+\ga_j$. This replacement
  decreases $m_1+\cdots +m_n$ by 1, contradicting the minimality of
  this sum.  Now (c) follows immediately from (b) and (d) follows from
  (c), since $(\ga,\ga_i)=(1/d)\sum m_j(\ga_j,\ga_i)>0$.

  Let us recall the positive roots of classical type:
  \begin{alignat*}{2}
    &A_{n-1}   && \Phi^+=\{\gre_p-\gre_q:1\leq p < q \leq n\}\\
    &B_n && \Phi^+=\{\gre_p \pm \gre_q: 1\leq p <q\leq n\} \cup
    \{\gre_p:1\leq p\leq n\}\\
    &C_n && \Phi^+ = \{\gre_p \pm \gre_q:1\leq p < q\leq n\} \cup \{
    2\gre_p:1\leq p\leq n\}\\
    &D_n \qquad && \Phi^+=\{\gre_p\pm \gre_q:1\leq p<q \leq n\}
  \end{alignat*}
  Based on these representations, we can make an additional
  observation about \eqref{e.zsum}:
  \begin{itemize}
  \item[(e)] Suppose that some $\alpha_i$ has a component of $
    \gre_s$ but $\alpha$ does not. Then $\alpha_i = \gre_r\pm \gre_s$
    for some $r<s$, and there exists $j$ such that $\alpha_j=\gre_r\mp
    \gre_s$. Moreover, $\alpha$ has a component of $\epsilon_r$ with
    positive coefficient.
  \end{itemize}
  Indeed, since $\alpha_i$ has a component of $ \gre_s$ but $\alpha$
  does not, some $\alpha_j$ must have a component of $\gre_s$ but with
  coefficient of opposite sign. This results in $(\ga_i,\ga_j)<0$,
  unless $\alpha_i=\gre_r\pm \gre_s$ and $\alpha_j=\gre_r\mp \gre_s$
  for some $r<s$. The final claim of (e) now follows from (d), which
  tells us that $(\alpha,\alpha_i)>0$ and $(\alpha,\alpha_j)>0$.

  If $\ga = \gre_p - \gre_q$ (in any type), then all terms on the
  right hand side of \eqref{e.zsum} must be of the form $\gre_a -
  \gre_b$ (since the sum of the coefficients of the $\gre_i$ is $0$).
  There must be one root $\ga_i$ on the right hand side of
  \eqref{e.zsum} of the form $\gre_r - \gre_q$, with $r \neq p$, since
  the coefficient of $\gre_q$ is negative.  Then since the coefficient
  of $\gre_r$ in the sum is $0$, there must be a root $\ga_j$ of the
  form $\gre_a - \gre_r$, but then $(\ga_i, \ga_j) = -1$, which is
  impossible by (c).  Therefore $\ga$ is not of the form $\gre_p -
  \gre_q$. This completes the proof for type $A_{n-1}$.

  If $\ga=\gre_p$ in type $B_n$ or $\ga=2\gre_p$ in type $C_n$, then
  some $\alpha_i$ must be of the form $\gre_p\pm \gre_s$, $s\neq
  p$. By (e), $p<s$ and some $\alpha_j$ must be of the form $\gre_p\mp
  \gre_s$. But this contradicts (a).  Therefore $\alpha$ is not of the
  form $\epsilon_p$ in type $B_n$ or $2\epsilon_p$ in type $C_n$.

  Hence, $\ga=\gre_p+\gre_q$ in type $B_n$, $C_n$, or $D_n$. Suppose
  that some $\alpha_i=\epsilon_p-\epsilon_q$. Then, since $\alpha$ has
  a component of $\epsilon_q$ with positive coefficient, some
  $\alpha_j$ must as well. But then $(\ga_i,\ga_j)<0$, contradicting
  (c). We conclude that no $\alpha_i$ equals $\gre_p-\gre_q$.

  Consider types $B_n$ and $C_n$.  Some $\alpha_i$ must have a
  component of $\epsilon_s$ which $\alpha$ does not.  (In type $C_n$,
  this is true because the only alternative is $\alpha_i=2\epsilon_p$,
  $\alpha_j=2\epsilon_q$ for some $i,j$, contradicting (a). In type
  $B_n$, a similar argument applies.)  Thus, by (e), either $\alpha_i$
  is of the form $\epsilon_p\pm \epsilon_s$ and there exists $j$ such
  that $\alpha_j$ is of the form $\epsilon_p\mp \epsilon_s$, or
  $\alpha_i$ is of the form $\epsilon_q\pm \epsilon_s$ and there
  exists $j$ such that $\alpha_j$ is of the form $\epsilon_q\mp
  \epsilon_s$.  Assume the former; the proof for the latter is
  similar. In type $B_n$, since $S$ is closed,
  $\epsilon_p=(1/2)\alpha_i+(1/2)\alpha_j \in S$. Thus $\alpha_i,
  \alpha_j$, two summands of the right hand side of \eqref{e.zsum},
  can be replaced by $\epsilon_p, \epsilon_p$.  With this replacement,
  $m_1|\ga_1|^2+\cdots +m_n|\ga_n|^2$ decreases by 2, contradicting
  the minimality of this quantity. In type $C_n$, $2\gre_p=\ga_i+\ga_j
  \in S$. Now $\alpha_i, \alpha_j$ can be replaced by the single root
  $2\epsilon_p$. This replacement decreases $m_1+\cdots + m_n$ by 1,
  contradicting the minimality of this sum. This completes the proof
  for types $B_n$ and $C_n$.

  This leaves only the possibility that $\alpha=\epsilon_p+\epsilon_q$
  in type $D_n$. We have seen that no $\alpha_i$ equals
  $\epsilon_p-\epsilon_q$. By (e), all of the $\alpha_i$ are of the
  form $\gre_p \pm \gre_a$ and $\gre_q \pm \gre_b$, where $a,b$ are
  not equal to either $p$ or $q$.  Note that we cannot have both
  $\gre_p + \gre_a$ and $\gre_q - \gre_a$ on the right hand side,
  since the inner product would be $-1$.  Similarly, we cannot have
  both $\gre_p - \gre_a$ and $\gre_q + \gre_a$ on the right hand side.
  All coefficients except those of $\gre_p$ and $\gre_q$ are $0$, so
  we conclude that on the right hand side, $\gre_p + \gre_a$ occurs
  iff $\gre_p - \gre_a$ occurs, and $\gre_q + \gre_b$ occurs iff
  $\gre_q - \gre_b$ occurs.  By minimality of the expression
  \eqref{e.zsum}, we conclude that this expression must have the form
\begin{equation} \label{e.zsum3}
2 (\gre_p + \gre_q) = (\gre_p + \gre_a) + (\gre_p - \gre_a)
+ (\gre_q + \gre_b) + (\gre_q - \gre_b)
\end{equation}
for $a,b$ not equal to $p$ or $q$, and $a \neq b$.  Now, $\gre_p +
\gre_a \in S$, and
$$
\gre_p + \gre_a = (\gre_p - \gre_q) + (\gre_a + \gre_q).
$$
Since $S$ is closed, at least one of the roots $\gre_p - \gre_q$ or
$\gre_a + \gre_q$ must be in $S$.  If $\gre_p - \gre_q \in S$, then
$\gre_p - \gre_b$ is in $S$ as it is the sum $(\gre_p - \gre_q) +
(\gre_q - \gre_b)$, where both summands are in $S$; but then
$$
\ga = \gre_p + \gre_q = (\gre_p - \gre_b) + (\gre_q + \gre_b),
$$
which contradicts the assumption that $\alpha$ is not bi-decomposable.
On the other hand, if $\gre_a + \gre_q \in S$, then
$$
\ga = \gre_p + \gre_q = (\gre_p - \gre_a) + (\gre_a + \gre_q),
$$
again contradicting the assumption that $\alpha$ is not
bi-decomposable.
\end{proof}

The proof of the implication (iii) $\Rightarrow$ (ii) in classical
types is clearly the most difficult part of this proof. The question
of whether this implication holds in exceptional types is open.

\begin{remark} \label{r.notequiv}
If $\Phi$ is not simply laced, then rational decomposability is not equivalent to
integral decomposability.  For example, suppose $\ga$ and $\gb$ are
simple roots with $\IP{\ga}{\gb^{\vee}} = -c$ and $\IP{\gb}{\ga^{\vee}} = -1$
for $c>0$.  If $\Phi$ is not simply laced, then there exist $\ga, \gb$
with $c \geq 2$.  If $x = s_{\ga} s_{\gb} s_{\ga}$ then 
$\pam = \{\ga, \ga+ \gb, (c-1) \ga + c \gb \}$.  The rationally indecomposable
roots in $\pam$ are $\ga$ and $(c-1) \ga + c \gb$, but if 
if $c \geq 2$, all roots in $\pam$ are integrally indecomposable.
\end{remark}

\begin{proposition} \label{p.Aind-isoind}
If $\rs=\pam$ and $\gP$ is of classical type, then $(S\zgw)\ua = \isnd{(S\zgw)}$.
\end{proposition}

\begin{proof}
  The inclusion $\subseteq$ holds because any $A$-indecomposable
  element is iso-indecomposable.  We prove the reverse inclusion.
  Suppose $\gg \in \isnd{(S\zgw)}$.  If $\gg$ is $A$-decomposable in
  $S\zgw$, then $\gg$ is $A$-decomposable in $S$, so by Theorem
  \ref{t.iso-gen}, $\gg$ is iso-decomposable in $S$. By Lemma
  \ref{l.iso-preserve}, $\gg$ is iso-decomposable in $\rs\zgw$, a
  contradiction.  Hence $\gg \in (S\zgw)\ua$, proving the reverse
  inclusion.
\end{proof}

Suppose that $\rs=\pam$.  One sees easily (see Lemma \ref{l.sam}(i))
that $(\rs\ua)\zgw \st (\rs\zgw)\ua$ and $(\rs\ua)\xgw \st (\rs\xgw)\ua$.
The following corollary shows that in classical types, all four of these sets
are equal.

\begin{corollary}\label{c.summary-class}
  If $\rs=\pam$ and $\gP$ is of classical type, then $(\rs\ua)\xgw =  (\rs\xgw)\ua = 
  (\rs\ua)\zgw = (\rs\zgw)\ua$.
  
\end{corollary}
\begin{proof}
By Corollary \ref{c.ind-opp} and Remark \ref{r.ind-opp},  $(\rs\ua)\xgw = (\rs\ua)\zgw$ and $(\rs\xgw)\ua=(\rs\zgw)\ua$.
The proof is completed by observing that
 $$
 (S\ua)\zgw =  (\isnd{S})\zgw = \isnd{(S\zgw)} =  (S\zgw)\ua,
 $$
 where the first equality is due to Theorem \ref{t.iso-gen}, the second to Corollary \ref{c.isoind}, and
 the third to Proposition \ref{p.Aind-isoind}.  
\end{proof}

\begin{remark} \label{r.significance}
This corollary implies that if $x$ is fixed and one wants to calculate $(S\zgw)\ua$ for multiple $z$,
it is not necessary to check indecomposability separately for each $z$.  Rather, one can compute
the set $S^A$ of $A$-indecomposable elements in $S$, and then intersect with $S\zgw$.
\end{remark}

Figure \ref{f.indecpam} summarizes the relationships we have found among
the five main types of indecomposability in $\pam$.  For general root
sets, rational indecomposability implies the other four types (see
Figure \ref{f.indec}). The other four implications of Figure
\ref{f.indecpam} are proved in Propositions 
\ref{p.int-iso}, \ref{p.invc}, and \ref{p.iso-bi-gen}.

\begin{figure}[tbh]
\centering
    \begin{tikzpicture}[scale=0.6]
      \matrix (m) [matrix of math nodes, row sep = 10em, column sep = 10
      em] %
      {& |(1)| S\uiq &\\
        |(2)| S\uq & |(3)| S\uis & |(4)| S\uiz\\
        & |(5)| S\uz & \\};%
%
      \path[line width=1.2pt,>=stealth,->] %
      (1) edge (4) %
      (2) edge (5) %
      (2) edge (1) %
      (5) edge (4);
      \path[line width=1.2pt,>=stealth,->] %
      (2.15) edge (3.165) %
      (3.195) edge node[auto] {\small $S =\pam \atop
        \text{classical type}$} (2.-15); %
      \path[line width=1.2pt,>=stealth,->] %
      (4) edge node[auto] {\small $S=\pam \atop
        \text{simply laced}$} (3); %
      \draw[line width=1.2pt,xshift=-2pt, yshift=-2pt,>=stealth,<-]
      (3)-- node[above,sloped]{\tiny $S=\pam$}(1);
%
      \draw[line width=1.2pt,xshift=-2pt, yshift=-2pt, >=stealth,->]
      (5)-- node[above,sloped]{\small $S\st \gP^+ \atop
        \text{simply laced}$}(3);
    \end{tikzpicture}
    \caption{Relationships among five types of
      indecomposability. Arrows represent inclusions. (See also Figure
      \ref{f.indec} on page~\pageref{f.indec}.)}
\label{f.indecpam}
\end{figure}
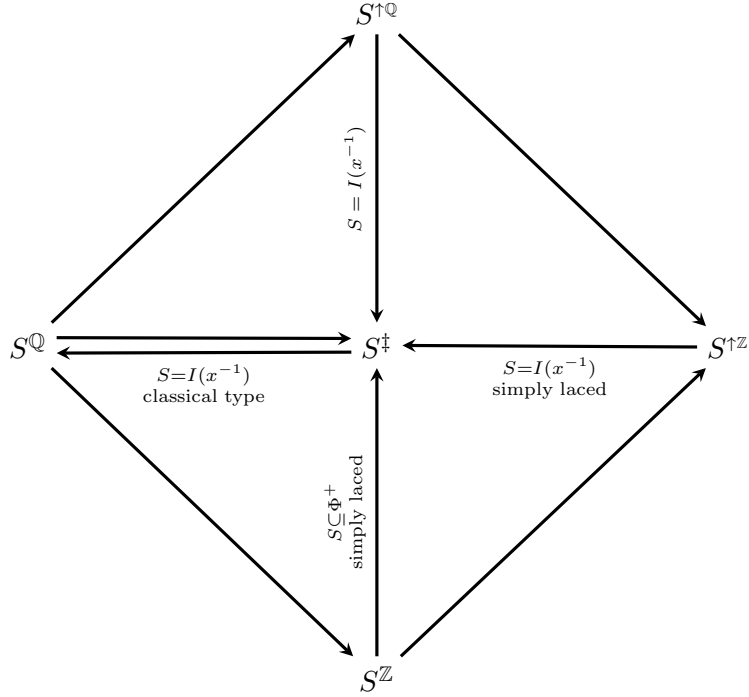

\subsection{Indecomposability in $\gP^+$}\label{ss.indgp}

When $S=\gP^+$, several types of indecomposability are easily proved
to be equivalent in all types. This is due to the following properties
of the base $\gD$ of $\gP^+$:
\begin{itemize}
\item[(a)] $\gD$ is linearly independent.
\item[(b)] Each root of $\gP^+$ is a non-negative integer linear
  combination of elements of $\gD$.
\item[(c)] Any root of $\gP^+$ not in $\gD$ can be written as a sum of two positive roots.
\end{itemize}

\begin{proposition}\label{p.equivs}
  Let $S=\gP^+$ and $\ga\in S$. The following conditions are
  equivalent:
  \begin{enumerate}
  \item $\ga$ is iso-decomposable.
  \item $\ga$ is bi-decomposable.
  \item $\ga\notin \gD$.
  \item $\ga$ is integrally decomposable.
  \item $\ga$ is rationally decomposable.  
  \end{enumerate}
\end{proposition}
\begin{proof}

  \noindent (iii) $\Rightarrow$ (iv) $\Rightarrow$ (v) $\Leftarrow$
  (i) are clear from the definitions, (ii) $\Leftarrow$ (iii) follows
  from (c), and (i) $\Leftarrow$ (ii) holds by Proposition
  \ref{p.iso-bi-gen}.

\noindent (iii) $\Leftarrow$ (v) 
Suppose that $\ga\in\gD$ is
rationally decomposable. Then $\ga=\sum_ic_i\gb_i$ for some $c_i\in
\Q_{>0}$ and $\gb_i\in\gP^+\setminus \{\ga\}$.  By (b), each $\gb_i$ in this sum can be expressed as
$  \gb_i=\sum_jd_{i,j}\ga_j$,
where $d_{i,j}\in \Z_{\geq}0$ and $\ga_j\in\gD$.
Thus $\ga = \sum_{i,j} c_id_{i,j} \ga_j$.

Suppose $j$ is such that $\ga_j\neq \ga$. By (a),
$\sum_ic_id_{i,j}=0$. Since $c_i>0$ and $d_{i,j}\geq 0$ for all $i$,
we must have $d_{i,j}=0$ for all $i$.
We conclude that $d_{i,j}=0$ for all $(i,j)$ such that $\ga_j\neq
\ga$. Hence $\gb_i$ is a positive
multiple of $\ga$, which is a contradiction.
\end{proof}

\begin{corollary}\label{c.indsimple}
  $(x\gP^-)\uis = x\gD^- = (x\gP^-)\uz=(x\gP^-)\uq$.
\end{corollary}
\begin{proof}
  This follows from Proposition \ref{p.equivs}, as $x$ is a
  length-preserving linear automorphism of $\gP$.
\end{proof}

\section{When all elements are indecomposable}\label{s.allin}

In this section we study root sets $S\st \gP^+$.  We are interested in
examining conditions under which all elements of $S$ are rationally,
integrally, or iso-indecomposable.  When this occurs, results in other
sections concerning such indecomposable elements of $\rs$ will of
course apply to all elements of $\rs$.

We say that $S$ is \textbf{coplanar} if there exists $v\in\ft$ such
that $\ga(v)=-1$ for all $\ga\in S$. Two cases are of particular
interest to us: if $\pam$ is coplanar, then $x$ is said to be a
\textbf{cominuscule} Weyl group element, and if $\gP(T_x\kxw)$ is
coplanar, then $x$ is called a \textbf{KL cominuscule point} of
$X^w$. Cominuscule Weyl group elements were first studied by Peterson,
and KL cominuscule points are studied in \cite{GrKr:23b} (see also
\cite{GrKr:19}).

The Weyl group element $x$ is said to be \textbf{fully commutative} if
there does not exist a reduced expression for $x$ which contains a
subword of the form $s_is_js_i\cdots$ of length $m\geq 3$, where $m$
is the order of $s_is_j$. In \cite{BJS:93}, Billey, Jockusch, and
Stanley showed that in type $A$, fully commutative Weyl group elements
can alternatively be characterized as 321-avoiding permutations, and
that their Schubert polynomials are flag skew Schur functions.  Full
commutativity was studied extensively by Fan and Stembridge in
\cite{Fan:95}, \cite{Fan:97}, \cite{FaSt:97}, \cite{Ste:96},
\cite{Ste:97}.

\begin{theorem}\label{t.fullcomm}
  If $x$ is a cominuscule Weyl group element, then $x$ is a KL
  cominuscule point and fully commutative.
\end{theorem}
\begin{proof}
  The first implication is due to $\gP(T_x\kxw)\st \pam$, and the
  second is proved  in \cite[Proposition 2.1]{Ste:01}.
\end{proof}

\begin{theorem}\label{t.allin}
  For $S\st \gP^+$, consider the following statements: 
  \begin{enumerate}
  \item $S$ is coplanar.
  \item $S\uz=S$.
  \item $S$ does not contain three roots of the form
    $\ga,\gb,\ga+\gb$.
  \item $S\uq = S$.
  \item $S\uis = S$.
  \item $x$ is fully commutative.  \end{enumerate} We have (i)
  $\Rightarrow$ (ii) $\Rightarrow$ (iii) and (iv) $\Rightarrow$
  (ii). If $\gP$ is simply laced, then (iii) $\Leftrightarrow$ (v). If
  $\gP$ is simply laced and $S=\pam$, then (i) $\Rightarrow$ (iv) and
  (iii) $\Leftrightarrow$ (vi). (See Figure \ref{f.allin}.)
\end{theorem}

\begin{figure}[tbh]
\centering
    \begin{tikzpicture}[scale=0.4]
      \matrix (m) [matrix of math nodes, row sep = 10em, column sep =
      6 em] %
      {|(1)| \text{(i) } S \text{ coplanar} & |(2)| \text{(ii) }S\uz =
        S & |(6)| \text{(vi) }x \text{
          fully commut.} \\
        |(4)| \text{(iv) }S\uq = S & |(3)| \text{(iii) }\nexists\,
        \ga,\gb,\ga+\gb\in S & |(5)| \text{(v) }S\uis=S\\};
      \path[line width=1.2pt,>=stealth,shorten >=8pt,shorten
      <=8pt,->]%
      (1) edge (2) %
      (4) edge (2) %
      (2) edge (3);
      \path[line width=1.2pt,>=stealth,shorten >=8pt,shorten
      <=8pt,<->]%
      (3) edge (5); %
      \draw[xshift=-2pt,yshift=-2pt,line width=1.2pt,>=stealth,shorten
      >=8pt,shorten <=8pt,->] %
      (1)-- node[above,sloped]{${S = \pam \atop
          \text{simply laced}}$} (4);
      \draw[xshift=-2pt,yshift=-2pt,shorten >=8pt,shorten <=8pt,line
      width=1.2pt,>=stealth] %
      (3)-- node[above]{\small simply laced} (5);
      \draw[xshift=-2pt,yshift=-2pt,line width=1.2pt,>=stealth,shorten
      >=8pt,shorten <=8pt,<->] %
      (3)-- node[above,sloped]{${S=\pam \atop \text{simply laced}}$}
      (6);
    \end{tikzpicture}
    \caption{To accompany Theorem \ref{t.allin}.  Arrows represent
      implications. Labelings of an arrow indicate conditions under
      which the implication or equivalence holds.}
\label{f.allin}
\end{figure}
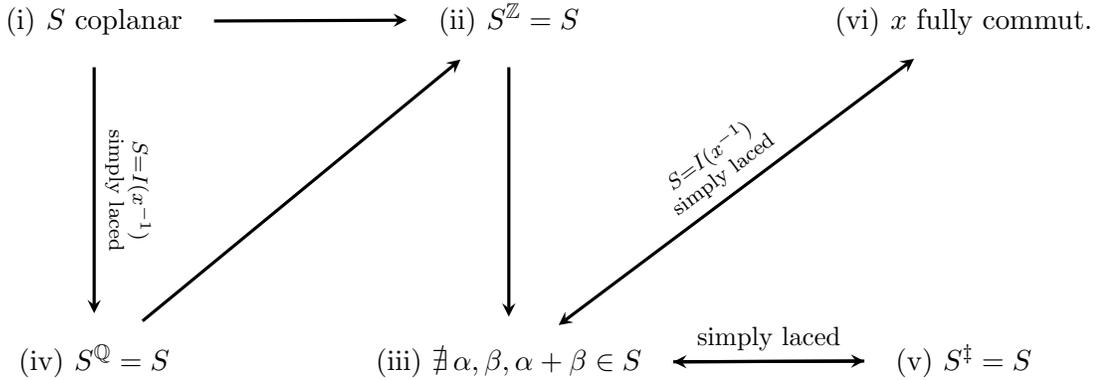

\begin{proof}
  (i) $\Rightarrow$ (ii) is proved in the same manner as
  \cite[Proposition 6.4]{GrKr:20}, but with $\pam$ replaced by $S$;
  (ii) $\Rightarrow$ (iii) is clear from definitions; (iv)
  $\Rightarrow$ (ii) since $S\uq\st S\uz$; (iii) $\Leftrightarrow$ (v)
  if $\gP$ is simply laced follows from Lemma \ref{l.simpiso}; (iii)
  $\Leftrightarrow$ (vi) if $\gP$ is simply laced and $S=\pam$ by
  \cite[Theorem 2.4]{FaSt:97}.

  \noindent (i) $\Rightarrow$ (iv) if $\gP$ is simply laced and
  $S=\pam$. Assume $S$ is coplanar. Then for all $\ga,\gb\in S$,
  $(\ga,\gb)\geq 0$. Indeed, if this were not true, then we would have
  $\ga+\gb\in\gP^+$. Since $S=\pam$ is closed under addition, this
  would imply $\ga+\gb\in S$, violating (iii).

  Suppose there exists $\gb\in S\setminus S\uq$. Let
  \begin{equation}\label{e.atech}
    \gb=\sum r_i\gb_i
  \end{equation}
  $r_i\in\Q_{>0}$, $\gb_i\in S\setminus \{\gb\}$ be a positive
  rational decomposition of $\gb$. Applying $(\, \cdot\, ,\gb_i)$ to
  both sides of \eqref{e.atech} yields $(\gb,\gb_i)>0$. Thus
  $\IP{\gb_i}{\gb}=1$.

  Applying $\IP{\cdot}{\gb}$ to both sides of \eqref{e.atech}, we
  obtain $2=\sum r_i$. Let $v$ be such that $\ga(v)=1$ for all $\ga\in
  S$.  Applying $v$ to both sides of \eqref{e.atech}, we find that
  $1=\sum r_i$, a contradiction.
\end{proof}

\begin{remark}
  Some of the conclusions of Theorem \ref{t.allin} (i) - (vi) can be
  strengthened. For example, (iii) $\Rightarrow$ (vi) does not require
  $S$ to be simply laced (although the converse does).
\end{remark}

\begin{corollary}\label{c.adind}
  If $\gP$ is of type $A$ or $D$, then all elements of $\pam$ are
  integrally indecomposable if and only if $x$ if fully commutative.
\end{corollary}
\begin{proof}
  In types $A$ and $D$, $\pam\uq=\pam\uis$, by Proposition
  \ref{p.iso-bi-gen}. Thus, in Theorem \ref{t.allin}, conditions (ii)
  through (vi) are equivalent (see Figure \ref{f.allin}).
\end{proof}

\begin{remark} \label{r.stem}
  In \cite[Remark 6.5]{GrKr:20}, the Weyl group element
  $x=s_2s_1s_3s_4s_2$ of type $D_4$ is considered. It is observed that
  all elements of $\pam$ are integrally indecomposable, but $x$ is not
  cominuscule (see \cite[Remark 5.4]{Ste:01}). Nevertheless, $x$ is
  fully commutative, as required by Corollary \ref{c.adind}.  Note that
Stembridge's results
apply since cominuscule Weyl group elements are exactly the minuscule elements
for the dual root system (see \cite[Section 5.2]{GrKr:19}).
\end{remark}


\section{Tangent spaces and $T$-invariant curves}
\label{s.schubert}

In this section we recall some known results on tangent spaces and
$T$-invariant curves of Schubert and Kazhdan-Lusztig varieties. We
include proofs for the convenience of the reader.

Let $G$ be a semisimple Lie group defined over an algebraically closed
field of characteristic 0, $B\supset T$ a Borel subgroup and maximal
torus respectively.  For any representation $V$ of $T$, we denote the
weights of $V$ by $\gP(V)$.  Let $W=N_G(T)/T$, the Weyl group of
$G$. Let $P$ be a parabolic subgroup containing $B$, $L$ be the Levi
subgroup of $P$ containing $T$, and $W_P=N_L(T)/T$, the Weyl group of
$L$. Each coset $uW_P$ in $W/W_P$ contains a unique representative of
minimal length; denote the set of minimal length coset representatives
by $W^P\subseteq W$. When $P=B$, $L=T$, and $W^P=W$.  Fix
$w\leq x\in W^P$.

\subsection{Schubert and Kazhdan-Lusztig
  varieties}\label{ss.unipotent}

Let $B^-$ be the opposite Borel subgroup to $B$, and let $P^-$ be the
opposite parabolic subgroup to $P$.  Let $U$, $U^-$, and $U_P^-$ be
the unipotent radicals of $B$, $B^-$, and $P^-$ respectively. Define
$U^-(x)=xU^-x^{-1}$ and $U_P^-(x)=xU_P^-x^{-1}$. The subgroups $U$,
$U^-$, $U_P^-$, $U^-(x)$, and $U_P^-(x)$ are unipotent, and are
$T$-equivariantly isomorphic to their Lie algebras, with which we
often identify them.  The weights of the first three of these under
the $T$ action are denoted by $\gP^+$, $\gP^-$, and $\gP_P^-$
respectively; the weights of the last two are then $x\gP^-$ and
$x\gP_P^-$ respectively.

Let $U_{\ga}$ denote the root subgroup of $G$ corresponding to
$\ga\in\gP$. For unipotent subgroups $V$ of the above paragraph, we
have that $U_{\ga}\subseteq V$ if and only if $\ga\in \gP(V)$;
otherwise $U_{\ga}\cap V$ is the identity.  Moreover, $V\cong
\prod_{\ga\in\gP(V)}U_{\ga}$.

\begin{lemma}\label{l.uup} 
  We have
  \begin{enumerate}
  \item   $\uup= \uub$.
  \item $\puup=\pam$.
  \item $\pupm = (\pupm\cap \gP^-) \sqcup \pam$.
  \end{enumerate}
\end{lemma}
\begin{proof}
  (i) See \cite[Section 3]{Knu:09}.

  \noindent (ii) $\puup=\gP(\uup)=\gP(\uub)=\puub=\pam$.

  \noindent (iii) $\pupm = (\pupm\cap \gP^-)\sqcup (\puup) =
  (\pupm\cap \gP^-) \sqcup \pam$.
\end{proof}

The variety $G/P$ is called a generalized flag variety. The torus $T$
acts on $G/P$ by left translations, with fixed points being the cosets
$zP$, $z\in W^P$.  The unipotent subgroup $\upm$ embeds as an open
subset of $G/P$ under the mapping $\upm\to \upm xP$.  Thus we may view
$\upm xP$ as an open affine space in $G/P$, with subspace $(\uup)xP$.
The Schubert variety $\xw\subseteq G/P$ is defined to be
$\overline{B^-wP}$, the Zariski closure of the $B^-$ orbit through
$wP$. The Kazhdan-Lusztig variety $\kxw\subseteq G/P$ is defined to be
$\overline{B^-wP}\cap BxP$.  Since $BxP = UxP = (\uup)xP$, we
can alternatively write
\begin{equation}\label{e.kxw}
  \kxw=  \xw\cap
  (U_P^-(x)\cap U)xP.
\end{equation}
We see that $\kxw$ is an affine subvariety of $(\uup)xP$.  The
$T$-fixed points of $\xw$ are the cosets $zP\in G/P$ such that $z\in
W^P$ and $z\geq w$, whereas $xP$ is the unique $T$-fixed point of
$\kxw$.

\subsection{$T$-invariant curves of Schubert and Kazhdan-Lusztig
  varieties}

If $Z$ is a variety with a $T$-action, then a $T$-invariant curve of
$Z$ is defined to be an irreducible curve $C$ which is closed in $Z$
and stable under the $T$-action, i.e., $t\cdot C \subseteq C$ for all
$t \in T$.  If $z$ is a $T$-fixed point of $Z$, then denote the set of
all $T$-invariant curves of $Z$ containing $z$ by $E(Z,z)$. Denote the
tangent space to $Z$ at $z$ by $T_zZ$, and
$\sum_{C\in E(Z,z)}T_zC\st T_zZ$ by
$TE_zZ$. 

For $\ga\in x\gP_P^-$, the Zariski closure of $U_{\ga}xP$ in $G/P$,
denoted by $\overline{U_{\ga}xP}$, is a $T$-invariant curve of $G/P$.
It is isomorphic to $\PP^1$ and has decomposition $U_{\ga}xP\cup
s_{\ga}xP$.  Note that $U_{\ga}xP\st (\uup)xP$ if and only if
$U_{\ga}\st \uup$ if and only if $\ga\in\gP(\uup)=\pam$.

\begin{proposition}\label{p.exw}
  Let $w\leq x\in W^P$. 
  \begin{enumerate}
  \item $E(\xw,xP)=\{\overline{U_{\ga}xP}\mid \ga\in x\gP_P^-,
    s_{\ga}x\geq w \}$.
  \item $E(\kxw,xP)=\{U_{\ga}xP\mid \ga\in \pam, s_{\ga}x\geq w \}$.
  \end{enumerate}
\end{proposition}
\begin{proof}
  (i) See \cite{Car:94}, \cite{CaKu:03}.

  \noindent (ii) Let $\ga\in x\gP_P^-$ be such that $s_{\ga}x\geq
  w$. By (i), $\overline{U_{\ga}xP}\st \xw$.  Since $\kxw=\xw\cap
  (U_P^-(x) \cap U)xP$,
\begin{equation*}
  \overline{U_{\ga}xP}\cap \kxw = 
\begin{cases}
  U_{\ga}xP, & \text{if }\ga\in x\gP_P^-\cap \gP^+=\pam\\
  xP, &\text{otherwise}
\end{cases}
\end{equation*}

Suppose that $\ga\in \pam$ and $s_{\ga}x\geq w$. Then
$U_{\ga}xP=\overline{U_{\ga}xP}\cap \kxw \in E(\kxw,xP)$. On the other
hand, suppose that $C\in E(\kxw,xP)$. Then $\overline{C}\in
E(X^w,xP)$, so $\overline{C}=\overline{U_{\gb}xP}$ for some $\gb\in
x\gP_P^-$ such that $s_{\gb}x\geq w$. Therefore $C=\overline{C}\cap
\kxw=U_{\gb}xP$, and $\gb\in\pam$.
\end{proof}

\begin{corollary}\label{c.exw}   Let $w\leq x\in W^P$. 
  \begin{enumerate}
  \item $\gP(TE_x\xw)=\{\ga\in x\gP_P^-\mid s_{\ga}x\geq w\}$. 

  \item $\gP(TE_x\kxw)=\{\ga\in\pam \mid s_{\ga}x\geq w\}$.
  \end{enumerate}
\end{corollary}
\begin{proof}
  For $\ga\in x\gP^-$, the tangent spaces to $\overline{U_{\ga}xP}$
  and $U_{\ga}xP$ at $x$ are both isomorphic to $U_{\ga}$, and thus
  both have weight $\ga$.  Hence (i) and (ii) follow from Proposition
  \ref{p.exw}(i) and (ii) respectively.
\end{proof}

\begin{lemma}\label{l.xgp}
  Let $w\leq x\in W^P$. Then $x\gP_P^-\cap \gP^- =\{\ga\in
  x\gP_P^-\mid s_{\ga} x>x\}$ and $x\gP_P^-\cap \gP^+=\pam = \{\ga\in
  x\gP_P^-\mid s_{\ga}x<x\}$.
\end{lemma}
\begin{proof}
  Observe that $s_{\ga}x>x$ if and only if $x^{-1}s_{\ga}>x^{-1}$ if
  and only if $\ga$ and $x^{-1}\ga$ have the same sign. Thus, if
  $\ga\in \pupm$, then $s_{\ga}x>x$ implies $\ga\in \pupm\cap \gP^-$,
  and $s_{\ga}x<x$ implies $\ga\in \pupm\cap \gP^+=\pam$.
\end{proof}

\begin{corollary}\label{c.tanwts}
  Let $w\leq x\in W^P$.
  \begin{enumerate}
  \item $\gP(T_x\xw)\st \pupm$.
  \item $\gP(T_x\xw)=(\pupm\cap \gP^-)\sqcup \gP(T_x\kxw)= \{\ga\in
    x\gP_P^-\mid s_{\ga}x>x\} \sqcup \gP(T_x\kxw)$.
  \item $\gP(T_x\kxw)=\gP(T_x\xw) \cap \pam$.
  \item $\gP(TE_x\xw)=(\pupm\cap \gP^-)\sqcup \gP(TE_x\kxw)$.
  \item $\gP(TE_x\kxw)=\gP(TE_x\xw)\cap \pam =\{\ga\in x\gP_P^-\mid x
    > s_{\ga}x\geq w\}$.
  \end{enumerate}
\end{corollary}
\begin{proof} 
  (i) $T_x\xw\st T_x(G/P) \cong \upm$, which has weights $\pupm$.

  \noindent (ii) The first equality appears in \cite[Lemma
  6.2(i)]{GrKr:20}; the second then follows from Lemma \ref{l.xgp}.

  \noindent (iii) follows from (i), (ii), and Lemma \ref{l.uup}(iii).

  \noindent (iv) By Lemma \ref{l.uup}(iii) and Corollary \ref{c.exw},
  $\gP(TE_x\xw)=\{ \ga\in x\gP_P^-\cap \gP^- \mid s_{\ga}x\geq w\}
  \sqcup \gP(TE_x\kxw)$. But by Lemma \ref{l.xgp}, for all $\ga\in
  x\gP_P^-\cap \gP^-$, $s_{\ga}x\geq w$. 

  \noindent (v) The first equality follows from (i), (iv), and Lemma
  \ref{l.uup}(iii). The second can be deduced from Corollary
  \ref{c.exw}(ii) and Lemma \ref{l.xgp}.
\end{proof}

\section{Indecomposable weights of tangent spaces and $T$-invariant
  curves}\label{s.ind-tinv}

Fix $w\leq x\in W^P$.  Define $\pta=\gP(T_x\xw)$, $\pcu=\gP(TE_x\xw)$,
$\ptak=\gP(T_x\kxw)$, and $\pcuk=\gP(TE_x\kxw)$.  We refer to
$\pta$ and $\ptak$ as sets of tangent weights, and $\pcu$ and $\pcuk$
as sets of curve weights.

In this section we prove the main result of this paper, Theorem \ref{t.main}
(see Corollary \ref{c.contain-sch}): $\pta \st
\conea \pcu$. We deduce this from the stronger statement $\ptak \st
\conea \pcuk$. This result relies on properties of $\ptak$ studied in
\cite{GrKr:20}, the characterization of $\pcuk$ by Carrell-Peterson
\cite{Car:94}, and Corollary \ref{c.ind-opp}.

\begin{proposition}\label{p.pcu}
  We have:
  \begin{enumerate}
  \item $\pcuk\subseteq \ptak\subseteq \pam$.
  \item $\pcuk = \pam\xgw$.
  \item $\ptak \st \conez (\pam\zgw)$.
  \end{enumerate}
\end{proposition}
\begin{proof}
  (i) The first inclusion holds because  $TE_x\kxw \st T_x\kxw$. The
  second follows from Corollary \ref{c.tanwts}(iii).

\noindent  (ii) is due to
  Corollary \ref{c.exw}(ii) and the equality $x_i=s_{\gg_i}x$. 

  \noindent (iii) We first give definitions of three terms which
  appear in equation \eqref{e.charc} below. Recall that
  $\vs=(s_1,\ldots,s_l)$ is a fixed reduced expression for $x$. Define
  $\ct_{w,\vs}$ to be the set of sequences $\vt=(i_1,\ldots,i_m)$,
  $1\leq i_{1}<\cdots < i_m\leq l$, such that $H_{s_{i_1}}\cdots
  H_{s_{i_m}}=H_w$. For such $\vt$, define $e(\vt)=m-\ell(w)$. For
  $\gz\in \conez\{\gg_i:i\notin \vt\}$, define $n_{\gz}$ to be the
  number of ways to express $\gz$ as a nonnegative integer linear
  combination of the $\gg_i$, $i\notin \vt$.

  Let $B$ and $C$ be the coordinate rings of the tangent space and
  tangent cone respectively of $\kxw$ at $x$, and let $B_1$ and $C_1$
  be the degree one components of these rings. By equation (5.2) of
  \cite{GrKr:20} (see also \cite[Theorem 6.1]{GrKr:20}) and the
  simplifications following this equation, the character of $C$ under
  the $T$ action is
\begin{equation}\label{e.charc}
  \Char C=\sum\nolimits_{\vt\in \ct_{w,\vs}}\sum\nolimits_{\gz\in
    \conez\{\gg_i: i\notin \vt\}}(-1)^{e(\vt)}n_{\gz}e^{-\gz}.
\end{equation}
For each $\gg_i$ in this summation, $i\notin \vt$ for some
$\vt\in\ct_{w,\vs}$; by \cite[Theorem 5.8]{GrKr:20}, $z_i\geq
w$. Therefore all weights $\gz$ of $C$ lie in $-\conez (\pam\zgw)$. So
too do all weights of $B_1$, since $B_1$ and $C_1$ are canonically
identified. But $B_1$ is the dual space of $T_x\kxw$. Hence all
weights of $T_x\kxw$ lie in $\conez (\pam\zgw)$.
\end{proof}

\begin{theorem}\label{t.containments} We have: 
  \begin{enumerate}
     \item $\conea \ptak = \conea \pcuk$, or equivalently,
     $(\ptak)\ua=(\pcuk)\ua$. Hence $\ptak \st \conea \pcuk$.
    \item $\ptak\cap \pam\ua=\pcuk\cap \pam\ua$.
    \item $\ptak\cap \pam\ua = \pcuk\cap \pam\ua = (\pcuk)\ua =
      (\ptak)\ua$, if $\gP$ is of classical type.
  \end{enumerate}
\end{theorem}
\begin{proof} 
  (i) For $\rs=\pam$,
 \begin{equation*}
   \ptak \st \conez (\rs\zgw) \st \conea (\rs\zgw) = \conea (\rs\xgw) 
   = \conea \pcuk,
 \end{equation*}
 where the inclusions and equalities are due, respectively, to
 Proposition \ref{p.pcu}(iii), $\Z\st\A$, Corollary \ref{c.ind-opp},
 and Proposition \ref{p.pcu}(ii). It follows that $\conea \ptak \st
 \conea \pcuk$, and the other inclusion is clear. The equivalence of
 the second equality of (i) is due to Corollary \ref{c.conez-zind}.

\noindent (ii) follows from (i) and Lemma \ref{l.indrel}(iii).

\noindent (iii) The first and third equalities are (ii) and (i)
respectively, and the second equality follows from Corollary
\ref{c.summary-class} and Proposition \ref{p.pcu}(ii).
\end{proof}

We can deduce from Theorem \ref{t.containments} analogous results for
Schubert varieties.

\begin{corollary}\label{c.contain-sch} We have:
  \begin{enumerate}
  \item $\conea \pta= \conea \pcu$, or equivalently,
    $(\pta)\ua=(\pcu)\ua$. Hence $\pta \st
\conea \pcu$.
  \item $\pta\cap (x \gP_P^-)\ua = \pcu\cap (x \gP_P^-)\ua$.
  \item $\pta\cap x\gD^- = \pcu\cap x\gD^-$, if $\xw\st G/B$.
 \end{enumerate}
\end{corollary}
\begin{proof} (i) follows from Theorem \ref{t.containments}(i) and
  Corollary \ref{c.tanwts}(ii), (iv).

\noindent (ii) follows from (i) and  Lemma \ref{l.indrel}(iii).

\noindent (iii) follows from (ii) and Corollary \ref{c.indsimple}.
\end{proof}

If we restrict attention to Schubert varieties in $G/B$ and $x=w_0$,
the longest element of the Weyl group, some of our results
simplify. In this case, since $P=B$, $\gP_P^-=\gP^-$ and
$x\gP_P^-=w_0\gP^-=\gP^+$.  Thus, by Corollaries
\ref{c.contain-sch}(i) and \ref{c.exw}(i),
$\pta \st \conea \pcu = \conea \{\ga\in \gP^+ \mid s_{\ga}w_0 \geq
w\}$.  Moreover, $s_{\ga}w_0\geq w$ is equivalent to
$s_{\ga}\leq w w_0$ (see \cite[Example 5.9.3]{Hum:90}).  Hence we
obtain
\begin{equation*}
  \pta \st  \conea \{\ga\in \gP^+ \mid s_{\ga}\leq  w w_0\}.  
\end{equation*}
The following proposition gives a stronger result for classical types.

\begin{proposition} \label{p.w0} Suppose $\Phi$ is of classical type,
  $P=B$, and $x = w_0$.  Then
  $(\pcu)\ua=\{\ga \in \gD\mid s_{\ga}\leq w w_0\}$, and thus
  \begin{equation}\label{e.wnot}
    \pta \st \conea \{\ga\in \gD\mid s_{\ga}\leq w w_0\}.
  \end{equation}
\end{proposition}

\begin{proof}
  Since $x = w_0$, $S = I(x^{-1}) = \Phi^+$.  By Proposition
  \ref{p.equivs}, $S^A = \gD$.  We have
  \begin{equation*}
    (S\xgw)^A = (S^A)\xgw = \{ \ga \in \gD \mid s_{\ga} w_0 \geq w \},
  \end{equation*}
  where the first equality is by Corollary \ref{c.summary-class}, and
  the second is because $S^A = \gD$. By Proposition \ref{p.pcu},
  $S\xgw = \pcuk$.  Since $x = w_0$, by Corollary \ref{c.tanwts}(iv),
  $\pcuk = \pcu$.  We conclude that
  \begin{equation*}
    (\pcu)\ua = \{ \ga \in \gD \mid s_{\ga} w_0 \geq w \}.
  \end{equation*}
  Since $s_{\ga}w_0\geq w \Leftrightarrow s_{\ga}\leq w w_0$, and
  $\pta \st \conea ((\pcu)\ua)$ (see Corollary \ref{c.sinco}), the
  result follows.
\end{proof}

We remark that the condition $s_{\ga}\leq w w_0$ appearing in
Proposition \ref{p.w0} is equivalent to the condition that the simple
reflection $s_{\ga}$ occurs in a reduced expression for $w w_0$.

\section{When tangent weights are curve weights}
\label{s.smoothness}

In this section we look at some consequences of Theorem
\ref{t.containments} and Corollary \ref{c.contain-sch} when
one knows that an element of $\pta$ must be contained in $\pcu$.  For example,
we obtain conditions on $x$ and $\gP$ which ensure that $\pta=\pcu$.
We also give smoothness criteria for $\xw$ which apply to such $x$ and
$\gP$.

\subsection{Characterizations of tangent spaces}\label{ss.charac}

Recall that $\pta\st \pupm$ and $\ptak\st \puup=\pam$.

\begin{theorem}(cf. \cite[Theorem A]{GrKr:20})\label{t.charac-kl-gb}
  Let $\ga\in\pam\ua$.  Then $\ga\in \ptak$ if and only if
  $s_{\ga}x\geq w$.
\end{theorem}
\begin{proof}
  This follows from Theorem \ref{t.containments}(ii) and Corollary
  \ref{c.exw}(ii).
\end{proof}

\begin{corollary}\label{c.schub-crit}
  Let $\ga\in \pupm$.
  \begin{enumerate}
  \item If $s_{\ga}x>x$, then $\ga\in \pta$.
  \item If $s_{\ga}x<x$, then $\ga\in \pam$.  In this case, if
    $\ga\in\pam\ua$, then $\ga\in\pta$ if and only if
    $s_{\ga}x\geq w$.
  \end{enumerate}
\end{corollary}
\begin{proof}
(i)   If $s_{\ga}x>x$, then $\ga\in \pupm\cap \gP^-$, and thus $\ga\in
  \pta$ by Corollary \ref{c.tanwts}(ii).

  \noindent (ii) If $s_{\ga}x<x$, then $\ga\in\pam$, and thus, by
  Corollary \ref{c.tanwts}(ii), $\ga\in \pta$ if and only if $\ga\in
  \ptak$. By Theorem \ref{t.charac-kl-gb}, if $\ga\in\pam\ua$, this
  occurs if and only if $s_{\ga}x\geq w$.
\end{proof}

\begin{corollary}\label{c.gb-schub}
  Let $X^w\st G/B$ and let $\ga\in x\gD^-$. Then $\ga\in\pta$
  if and only if $s_{\ga}x\geq w$. 
\end{corollary}
\begin{proof}
  This follows from Corollaries \ref{c.contain-sch}(iii) and
  \ref{c.exw}(i).
\end{proof}

Recall from Section \ref{s.allin} that the element $x$ is said to be
{\em cominuscule} if there exists $v\in \ft$ such that $\ga(v)=-1$ for
all $\ga\in I(x^{-1})$, and it is said to be a {\it KL cominuscule
  point of $X^w$} if there exists $v\in\ft$ such that $\ga(v)=-1$ for
all $\ga\in\ptak$.  The maximal parabolic subgroup $P\supseteq B$ (or
sometimes $G/P$) is said to be {\em cominuscule} if the simple root
corresponding to $P$ occurs with coefficient 1 when the highest root
of $G$ is written as a linear combination of the simple roots.

\begin{theorem}\label{t.comin-char}
  Suppose that $\gP$ is simply laced and that any of the following hold:
  \begin{enumerate}
  \item All elements of $\ptak$ are integrally indecomposable in
    $\ptak$.
  \item All elements of $\pam$ are integrally indecomposable in
    $\pam$.
  \item $x$ is fully commutative and $\gP$ is of types $A$ or $D$.
  \item $x$ is a KL cominuscule point of $X^w$.
  \item $x$ is cominuscule.
  \item $P$ is cominuscule.
  \end{enumerate}
  Then $\ptak=\pcuk$ and $\pta = \pcu$. If $\ga\in\pupm$, then
  $\ga\in\pta$ if and only if $s_{\ga}x\geq w$, and
  $\ga\in\ptak$ if and only if $x>s_{\ga}x\geq w$.
\end{theorem}
\begin{proof}
  (iii) $\Rightarrow$ (ii) $\Rightarrow$ (i) by Corollary
  \ref{c.adind} and Lemma \ref{l.indrel}(v) respectively; (vi)
  $\Rightarrow$ (v) $\Rightarrow$ (iv) $\Rightarrow$ (i) by
  \cite[Proposition 6.7]{GrKr:20}, definition, and Theorem
  \ref{t.allin} respectively.

  Thus we may assume that (i) holds, i.e., $(\ptak)\uz=\ptak$. By
  Lemma \ref{l.indrel}(v), $(\pcuk)\uz=\pcuk$.  Hence Theorem
  \ref{t.containments}(i) implies $\ptak=\pcuk$. From Corollary
  \ref{c.tanwts}(ii) and (iv) it follows that $\pta = \pcu$.  Let
  $\ga\in x\gP_P^-$. Then $\ga\in\ptak =\pcuk \Leftrightarrow
  x>s_{\ga}x\geq w$, by Corollary \ref{c.tanwts}(v).  By Corollary
  \ref{c.tanwts}(ii), $\ga\in \pta\Leftrightarrow s_{\ga}x>x \text{ or
  }x>s_{\ga}x\geq w$ $\Leftrightarrow s_{\ga}x\geq w$.
\end{proof}

\subsection{Smoothness criteria}\label{ss.smoothness}

Recall that $\vs=(s_1,\ldots,s_l)$ is a reduced expression for
$x$. Since $w\leq x$, $\vs$ contains a reduced subexpression for $w$.
Let $M=\{i\in [l]: (s_1,\ldots,\wh{s}_i,\ldots,s_l)$ contains a
reduced subexpression for $w\}$.

\begin{lemma}\label{l.lxw}
  $|M|=\ell(x)-\ell(w)$ if and only if $\vs$ contains a unique reduced
  subexpression for $w$.
\end{lemma}
\begin{proof}
  Let $1\leq i_1< \cdots <i_t\leq l $ be such that
  $(s_1,\ldots,\wh{s}_{i_1},\ldots,\wh{s}_{i_t},\ldots,s_l)$ is a
  reduced expression for $w$.  If it is the unique reduced
  subexpression of $\vs $ for $w$, then $M=\{i_1,\ldots, i_t\}$, and
  $|M|=l-\ell(w)=\ell(x)-\ell(w)$.  If there exists another reduced
  subexpression of $\vs$ for $w$, then $M\supsetneq
  \{i_1,\ldots,i_t\}$, and $|M|>\ell(x)-\ell(w)$.
\end{proof}

If $Z$ is any scheme, then $z\in Z$ is \textbf{smooth} if
$\dim T_zZ= \dim Z$ (see \cite[Theorem 4.2.1]{BiLa:00}).

\begin{theorem}\label{t.smooth}
  Suppose that $\gP$ is simply laced.  If any of the conditions of
  Theorem \ref{t.comin-char} are satisfied, then the following are
  equivalent:
  \begin{enumerate}
  \item $x$ is a smooth point of $X^w$.
  \item $x$ is a smooth point of $\kxw$.
  \item $|\{\ga \in x\gP_P^-:s_{\ga}x\geq w\}|=\dim \xw$.
  \item $|\{ \ga\in\pam : s_{\ga}x\geq w\}| = \dim \kxw$.
  \end{enumerate}
  If conditions (ii), (iii), (v), or (vi) of Theorem
  \ref{t.comin-char} are satisfied, then the four equivalent
  statements above are equivalent to
  \begin{enumerate}
  \item[(v)] $\vs$ contains a unique reduced subexpression for $w$.
  \end{enumerate}
\end{theorem}
\begin{proof}
  (i) $\Leftrightarrow$ (ii) This is because locally, $\xw$ is the
  product of $\kxw$ with the affine space $U_P^-(x)\cap U^-$ (see
  \cite[Lemma 5.1(ii)]{GrKr:20}).

  \noindent (i) $\Leftrightarrow$ (iii) By Theorem \ref{t.comin-char},
  $\dim T_x\xw = |\pta| = |\pcu| = |\{\ga\in x\gP_P^- : s_{\ga}x\geq
  w\}|$.

  \noindent (ii) $\Leftrightarrow$ (iv) By Theorem \ref{t.comin-char},
  $\dim T_x\kxw = |\ptak| = |\pcuk| = |\{\ga\in \pam : s_{\ga}x\geq
  w\}|$.

  \noindent (iv) $\Leftrightarrow$ (v) if conditions (ii), (iii), (v),
  or (vi) of Theorem \ref{t.comin-char} are satisfied.  For any
  $\ga\in \pam$, $\ga=\gg_i$ for some $i\in [l]$, and $s_{\ga}x=x_i$.
  Thus
  \begin{equation*}
    |\{\ga\in
    \pam : s_{\ga}x\geq w\}| = |\{ i\in [l]:x_i\geq w\}|
  \end{equation*}

  Conditions (iii), (v), and (vi) of Theorem \ref{t.comin-char} all
  imply condition (ii) of the same theorem; hence we may assume that
  $\pam\uz = \pam$. By Theorem \ref{t.allin}, $\pam\uis = \pam$; hence
  $x_i=z_i$ for all $i$. Consequently,
  \begin{equation*}
    \{ i\in [l]:x_i\geq w\} = \{ i\in [l]:z_i\geq w\} 
  \end{equation*}

  By Lemma \ref{l.heckefact}, $z_i\geq w$ if and only if
  $(s_1,\ldots,\wh{s}_i,\ldots, s_l)$ contains a reduced subexpression
  for $w$. Therefore $\{ i\in [l]:z_i\geq w\} = M$, so (iv) is
  equivalent to $|M| = \ell(x) - \ell(w)$. The result now follows from
  Lemma \ref{l.lxw}.
\end{proof}

For $P$ cominuscule, the equivalence of (i) and (v) of Theorem
\ref{t.smooth} holds even when $\gP$ is not simply laced.  This can be
deduced from \cite[Corollary 2.11]{GrKr:15}, which states that the
multiplicity of $x\in X^w$ when $P$ is cominuscule is equal to the
number of reduced subexpressions of $\vs$ for $w$. Of course, $x$ is a
smooth point of $X^w$ precisely when its multiplicity equals 1. Thus
the result is obtained. Further discussion of the criterion (v) for
smoothness appears in \cite{GrKr:23b}.

\section{Examples} \label{s.examples} In this section we consider
examples.  In Section \ref{ss.3D}, we use our main result to determine
the tangent space at $w_0$ to singular 3-dimensional Schubert
varieties for an irreducible root system not of type $G_2$.  Section
\ref{ss.ex_typeD} focuses on type $D$.  We define a family of elements
$w_{ab} \in W$, for $1 \leq a < b < n-1$, and let $\pta$ and $\pcu$ be
defined taking $x = w_0$ and $w = w_{ab}$.  We show that $\pta$
properly contains $\pcu$, and verify by direct calculation that
$\conea \pcu \supseteq \pta$, as guaranteed by our main result Theorem
\ref{t.main}.

Throughout this section, we take $P = B$ so $W^P = W$.  In Sections
\ref{ss.3D} and \ref{ss.ex_typeD}, we fix $x = w_0$.  The calculations
in our examples use the descriptions of the tangent spaces to Schubert
varieties given in \cite{Lak:00}, \cite{BiLa:00}, as well as a result
from \cite{Bri:98}.  In this paper, we consider the Schubert varieties
$X^w = \overline{B^- \cdot wB}$, whereas those references consider the
opposite Schubert varieties of the form $X_w = \overline{B \cdot wB}$.
However, as is well-known, Schubert varieties and opposite Schubert
varieties are isomorphic.  The relation between tangent spaces is
given by the following lemma, whose proof we omit.

\begin{lemma} \label{l.translation} We have $w_0 X_{w_0 w} = X^w$, and
  $w_0 \Phi(T_{w_0 x} X_{w_0 w} ) = T_x X^w$.
\end{lemma}

We will use this lemma without comment and state the results from
\cite{Lak:00}, \cite{BiLa:00} and \cite{Bri:98} in terms of Schubert
varieties of the form $X^w$.

We will use some general facts about the relation of $\pcu$ and
$\pta$.  In general, we have $\pcu \subseteq \pta$.  Also, Carrell and
Peterson (\cite{Car:94}; see also \cite[Cor.~19 (iv)]{Bri:98}) proved
that $| \pcu | \geq \dim X$.  In addition,
$\dim T_x X = | \pta | \geq \dim X$, and $X$ is smooth at $x$ if and
only if this inequality is an equality (see e.g. \cite[Theorem
4.2.1]{BiLa:00}).

Part (1) of the following lemma is well-known, and part (2) is a fact
which will be useful below.

\begin{lemma} \label{l.general} Let $\pcu$ and $\pta$ be the curve and
  tangent weights of $X^w$ at $x$.

$(1)$ If $X^w$ is smooth at $x$, then $\pcu = \pta$.

$(2)$ If $X^w$ is not smooth at $x$ and $| \pcu | = \dim X$, then
$\pcu \neq \pta$.
\end{lemma}

\begin{proof}
  If $X^w$ is smooth at $x$, then
$$
\dim X^w \leq | \pcu | \leq | \pta | = \dim X^w,
$$
so the middle inequality must be an equality and hence $\pcu = \pta$.
This proves (1).  If $X^w$ is not smooth at $x$, then and
$| \pcu | = \dim X$, then $| \pcu | < | \pta |$, implying (2).
\end{proof}

\subsection{Three-dimensional Schubert varieties} \label{ss.3D} In
this section, $\pta$ and $\pcu$ always refer to the sets of curve and
tangent weights at the point $x = w_0$.  The following lemma is mainly
a rephrasing of results in \cite{Bri:98}, but we state it for
convenience.
 
\begin{lemma} \label{l.3D} A $3$-dimensional Schubert variety $X^w$ is
  singular at $w_0$ if and only if $w = w_0 s_{\ga} s_{\gb} s_{\ga}$
  for nonorthogonal simple roots $\ga$, $\gb$ with
  $\IP{\gb}{\ga^{\vee}} \leq -2$.  For such $w$, $X^w$ is rationally
  smooth at $w_0$.
 \end{lemma}
 
 \begin{proof}
   Schubert varieties of dimension $3$ are of the form $X^w$, where
   either:
 \begin{myenumerate}
 \item $w = w_0 s_{\ga} s_{\gb} s_{\ga}$ for nonorthogonal simple
   roots $\ga$, $\gb$.
 \item $w = w_0 s_{\ga} s_{\gb} s_{\gg}$, where $\ga$, $\gb$, $\gg$
   are distinct simple roots.
\end{myenumerate}
For $w$ as in (1), the statement of the lemma is given in
\cite{Bri:98}.  To complete the proof, it suffices to show that if $w$
is as in (2), then $X^w$ is smooth at $w_0$.  This is most easily seen
by translating into the equivalent statement that
$X_{s_{\ga} s_{\gb} s_{\gg}}$ is smooth at $1$.  A formula for
equivariant multiplicities due to Arabia and Rossmann
(\cite[Prop.~3.3.1]{Ara:89} and \cite{Ros:89}; see \cite[Section
4]{Bri:98}) implies that the equivariant multiplicity of
$X_{s_{\ga} s_{\gb} s_{\gg}}$ at $1$ is $1/\ga \gb \gg$.  By
\cite{Kum:96} (see \cite[Cor.~19]{Bri:98})), this implies that $X^w$
is smooth at $w_0$, as desired.
\end{proof} 

\begin{lemma} \label{l.2root} Assume $\Phi$ is irreducible.  Suppose
  $\ga$ and $\gb$ are simple roots with
  $\IP{\gb}{\ga^{\vee}} \leq -2$, and let
  $w = w_0 s_{\ga} s_{\gb} s_{\ga}$.  Then
$$
\pcu = \{ \ga, \gb, 2 \ga + \gb \}.
$$
\end{lemma}

\begin{proof}
  Since there are two root lengths in $\Phi$, the action of $w_0$ on
  $\Phi$ is by multiplication by $-1$ (cf.~\cite[Sec.~13,
  Ex.~5]{Hum:72}), so $w_0$ is in the center of $W$.

  Taking $x = w_0$, we have
  $\Phi_{\cur} = \{ \gg \in \Phi^+ \mid s_{\gg} w_0 \geq w \}$.
  Equivalently,
  $\Phi_{\cur} = \{ \gg \in \Phi^+ \mid s_{\gg} \leq s_{\ga} s_{\gb}
  s_{\ga} \}$.  (The statements are equivalent because
  $s_{\gg} w_0 \geq w_1$ is equivalent to
  $w_0 s_{\gg} w_0 \leq w_0 w$; since $w_0$ is in the center of $W$,
  this is equivalent to $s_{\gg} \leq s_{\ga} s_{\gb} s_{\ga}$.)
  Since the length of $s_{\gg}$ is odd, the only possibilities for
  $s_{\gg}$ are $s_{\ga}$, $s_{\gb}$, or
  $s_{\ga} s_{\gb} s_{\ga} = s_{s_{\ga} \gb}$.  We conclude that
  $\pcu = \{ \ga, \gb, s_{\ga} \gb \} = \{ \ga, \gb, 2 \ga + \gb \}$.
\end{proof}

In light of Lemma \ref{l.3D}, the following proposition describes the
tangent space at $w_0$ to a singular $3$-dimensional Schubert variety
for an irreducible root system $\Phi$ not of type $G_2$.

\begin{proposition} \label{p.bc2} Assume $\Phi$ is irreducible.
  Suppose $\ga$ and $\gb$ are simple roots with
  $\IP{\gb}{\ga^{\vee}} = -2$.  Let $\pta$ and $\pcu$ correspond to
  $w = w_0 s_{\ga} s_{\gb} s_{\ga}$ and $x = w_0$.  Then
$$
\pta = \{ \ga, \gb, \ga + \gb, 2 \ga + \gb \}.
$$
\end{proposition}

\begin{proof}
  We have
  \begin{equation} \label{e.bc2} \pcu \subsetneq \pta \subseteq \conea
    \pcu;
  \end{equation}
  the first inclusion is proper by Lemma \ref{l.general} because
  $X^{w}$ is singular at $w_0$, and the second inclusion is our main
  result.  In this case, $\ga$ and $\gb$ span a root system of type
  $B_2$ (which is isomorphic to the root system of type $C_2$), with
  long root $\gb$.  Thus, there are $4$ roots in the cone spanned
  (over $\Q$ or $\Z$) by $\ga, \gb$, namely,
  $ \ga, \gb, s_{\ga} \gb = 2 \ga + \gb, s_{\gb} \ga = \ga + \gb $,
  the first three of which are in $\pcu$.  From \eqref{e.bc2}, we
  conclude that $\pta$ must consist of all $4$ of these roots.
\end{proof}
 
 Note that for $\Phi$ of type $G_2$ (which occurs exactly when
 $\IP{\gb}{\ga^{\vee}} = -3$), there are more than $4$ positive roots,
 so the arguments above do not suffice to determine $\pta$.
 
 \subsection{Examples in type $D_n$} \label{ss.ex_typeD} In this
 section we assume $\Phi$ is of type $D_n$.  We define a family of
 elements $w_{ab}$ ($a<b<n-1$), and consider the tangent spaces
 $T_{w_0} X^{w_{ab}})$ at $x = w_0$.  We write
 $\Phi_{\oth} = \Phi_{\tan} \setminus \Phi_{\cur}$ for the set of
 ``other" roots in the tangent space of $X^{w_{ab}}$ at $w_0$.  (Of
 course, $\pta, \pcu$, and $\Phi_{\oth}$ all depend on the choice of
 $x = w_0$ and of $w_{ab} \in W$, but we omit this from the notation.)
 We will calculate $\Phi_{\oth}$, and identify enough elements of
 $\pcu$ to show that each root in $\Phi_{\oth}$ is a sum of two roots
 in $\pcu$.  Thus, in this example, we can verify our main result that
 $\pta \subseteq \conea \pcu$ by direct calculation.
 
 We have chosen $x = w_0$ in this section because the description of
 the tangent spaces $T_x (X^w)$ in \cite{Lak:00} and \cite{BiLa:00} is
 much simpler for $x=w_0$ than for arbitrary $x$.
 
 We use the standard realization of the root system of $D_n$ as in
 Section \ref{s.dec-iso-dec}.  Let $w \in W$.  We have
 $\pcu = \{ \gg \in \Phi^+ \mid s_{\gg} w_0 \geq w \}$.  Applying
 \cite[Theorem 6.8]{Lak:00} or \cite[Theorem 5.3.1]{BiLa:00}, we see
 that $\gg \in \Phi_{\oth}$ if and only if $\gg = \gre_i + \gre_j$,
 with $1 \leq i < j < n-1$, and

$(1)$ $s_{\gg} w_0 \not\geq w$, and

$(2)$ $s_{\gre_i - \gre_n} s_{\gre_i + \gre_n} s_{\gre_j + \gre_{n-1}} w_0 \geq  w$.

\noindent (In translating the result from \cite{Lak:00} we have used
the fact that $w_0$ commutes with
$s_{\gre_i - \gre_n} s_{\gre_i + \gre_n} s_{\gre_j + \gre_{n-1}} $, as
can easily be seen by writing the elements of $W$ as signed
permutations (cf.~\cite{BoGr:03}.)

Motivated by condition (2), for $1 \leq a < b < n-1$, we define
$u_{ab} = s_{\gre_a - \gre_n} s_{\gre_a + \gre_n} s_{\gre_b +
  \gre_{n-1}}$, and set $w_{ab} = u_{ab} w_0$.  If $w = w_{ab}$, then
we claim that (1) and (2) are equivalent to the conditions:

$(1')$ $s_{\gg} \not\leq u_{ab}$, and

$(2')$ $u_{ij} < u_{ab}$.

Indeed, condition (1) states that
$s_{\gg} w_0 \not\geq w_{ab} = u_{ab} w_0$, which is equivalent to
$(1')$.  Condition (2) states that
$u_{ij} w_0 \geq w_{ab} = u_{ab} w_0$, which is equivalent to $(2')$.
This verifies the claim.

\begin{proposition}
  Suppose $\Phi$ is of type $D_n$ for $n \geq 4$.  Fix $a, b$
  satisfying $1 \leq a < b < n-1$.  Let $\pta$, $\pcu$, and
  $\Phi_{\oth}$ be defined as above, corresponding to $x = w_0$ and
  $w = w_{ab}$.  Assume below that $i,j$ denote integers satisfying
  $1 \leq i < j < n-1$. Then:

$(a)$ $\Phi_{\oth} = \{ \gre_i + \gre_j \mid  i \geq a , j \geq b \}$.

$(b)$ Suppose $i \geq a$ and $ j \geq b$.  Then the roots
$\gre_i \pm \gre_{n-1}$ and $\gre_j \pm \gre_{n-1}$ are in $\pcu$.

$(c)$ $\pta \subseteq \conea \pcu$.
\end{proposition}

\begin{proof} We sketch the proof, but omit most details, which
  involve calculations in the Bruhat order in type $D_n$.  We adopt
  the conventions and notation of \cite{BoGr:03}.  The Weyl group of
  type $D_n$ can be realized as the group of signed permutations of
  $1, \ldots, n$, with an even number of negative signs.  If
  $u \in W$, write $u(i) = u_i$; then $u$ can be represented by the
  sequence $u_1 u_2 \ldots u_n$.  Here
  $u_i \in \{ 1, \bar{1}, 2, \bar{2}, \ldots, n, \bar{n} \}$, where
  $\bar{a}$ denotes $-a$.  Write $\gg = \gre_i + \gre_j$.

  The proof uses two key facts about the Bruhat order.  The first is
  the fact that if $u \in W$ and $\gb$ is a positive root with
  $u \gb > 0$, then $u < u s_{\gb}$.  The second is a characterization
  of the Bruhat order in type $D_n$ in terms of the sequences
  representing elements of $W$ as signed permutations.  This
  characterization is due to Proctor \cite{Pro:82}; the statement may
  also be found in \cite[Prop.~2.10]{BoGr:03}.

  We first consider (a).  We know that $\gg \in \Phi_{\oth}$ if and
  only if conditions $(1')$ and $(2')$ are satisfied.  We claim first
  that $(2')$ is satisfied $\Leftrightarrow$ $a \leq i$ and
  $b \leq j$.  The implication $(\Rightarrow)$ is proved by
  considering the contrapositive.  If either of $i \geq a$ or
  $ j \geq b$ is false, then applying Proctor's condition to the
  signed permutations corresponding $u_{ij}$ and $u_{ab}$, we see that
  $u_{ij} \not\leq u_{ab}$.  The calculation is made easier because
  the parity condition in this characterization is not needed, and one
  can restrict attention to the places where the expressions for
  $u_{ij}$ and $u_{ab}$ differ from the identity permutation
  $12\ldots n$ (cf.~\cite[Lemma 3.4]{BoGr:03}).  To prove
  $(\Leftarrow)$, we suppose $i \geq a$ and $ j \geq b$.  In this
  case, we can exhibit a sequence of positive roots
  $\gb_1, \ldots, \gb_r$ such that
  \begin{equation} \label{e.sequence} u_{ij} < u_{ij} s_{\gb_1} <
    \cdots < u_{ij} s_{\gb_1} s_{\gb_2} \cdots s_{\gb_r} = u_{ab}.
\end{equation}
This calculation is facilitated by considering elements of $W$ as
signed permutations and using the description of multiplication by
reflections in \cite[(2.1)]{BoGr:03}.  We omit further details.  This
proves the claim.

To complete the proof of (a), we need to verify that if
$\gg = \gre_i+\gre_j$, with $a \leq i$ and $b \leq j$, then
$s_{\gg} \not\leq u_{ab}$.  This can be verified by writing down the
expressions for $s_{\gg}$ and $u_{ab}$ as signed permutations, and
again using Proctor's characterization of the Bruhat order.  We omit
further details.  This completes the proof of (a).

We next consider (b).  We know a root $\zeta$ is in $\pcu$ if and only
if $s_{\zeta} w_0 \geq w_{ab} = u_{ab} w_0$, or equivalently,
$s_{\zeta} < u_{ab}$.  Thus, we want to show that if $\zeta$ is one of
the four roots listed in the statement of (b), then
$s_{\zeta} < u_{ab}$.  Since
$s_{\gre_j - \gre_{n-1}} < s_{\gre_i - \gre_{n-1}} $ and
$s_{\gre_j + \gre_{n-1}} < s_{\gre_i + \gre_{n-1}} $, we only need to
prove this for $\zeta$ equal to $\gre_i - \gre_{n-1}$ or
$\gre_i + \gre_{n-1}$.  We prove that $s_{\zeta} < u_{ab}$ in the same
way the statement $u_{ij} \leq u_{ab}$ was proved in part (a), by
exhibiting sequences analogous to \eqref{e.sequence}.  We omit further
details.

Finally, part (c) follows from parts (a) and (b).  Indeed, since
$$
\gre_i + \gre_j = (\gre_i - \gre_{n-1}) + (\gre_j + \gre_{n-1}) =
(\gre_i + \gre_{n-1}) + (\gre_j - \gre_{n-1}),
$$ 
parts (a) and (b) imply that every root in $\Phi_{\oth}$ is in
$\conea \pcu$.
\end{proof}


\bibliographystyle{amsalpha}
\bibliography{removableroots}

\end{document}